\documentclass{article}
\pdfoutput=1
\usepackage[utf8]{inputenc}
\usepackage{amsmath}
\usepackage{amsfonts}
\usepackage{amssymb}
\usepackage{tikz-cd}
\usepackage{color,soul}
\usepackage{array}
\usepackage{zref-savepos}
\usepackage{amsthm}
\usepackage{multirow, bigstrut}
\usepackage[font={small,it}]{caption}
\usepackage{tikz}
\usepackage{bm}
\usepackage{mathabx}
\usetikzlibrary{arrows}
\usepackage{dynkin-diagrams}
\usepackage{adjustbox}

\usepackage{xurl}
\usepackage{hyperref}
\hypersetup{
    colorlinks=true,
    linkcolor=blue,
    filecolor=magenta,  
    urlcolor=cyan,
    breaklinks=true
}

\usetikzlibrary{calc}
\usepackage{comment}
\usepackage{enumerate}
\usepackage{tikz}
\usetikzlibrary{shapes.geometric}
\usepackage[backend=biber,
style=alphabetic,
bibencoding=ascii
]{biblatex}
\newtheorem{theorem}{Theorem}[section]
\newtheorem*{theorem*}{Theorem}
\newtheorem{lemma}[theorem]{Lemma}
\newtheorem{proposition}[theorem]{Proposition}
\newtheorem{corollary}[theorem]{Corollary}

\theoremstyle{definition}
\newtheorem{definition}[theorem]{Definition}
\newtheorem{remark}[theorem]{Remark}
\newtheorem{example}[theorem]{Example}

\theoremstyle{plain}

\newcommand{\R}{\mathbb{R}}

\newcommand{\Z}{\mathbb{Z}}

\newcommand{\Q}{\mathbb{Q}}

\DeclareMathOperator{\SO}{SO}
\DeclareMathOperator{\Sp}{Sp}
\DeclareMathOperator{\GL}{GL}
\DeclareMathOperator{\frob}{frob}
\DeclareMathOperator{\Lie}{Lie}
\DeclareMathOperator{\cusp}{cusp}
\DeclareMathOperator{\der}{der}
\DeclareMathOperator{\naive}{naive}

\DeclareMathOperator{\Span}{Span}

\newcommand{\bP}{\mathcal{P}}

\title{Inequalities characterizing distinguished unipotent orbits}
\author{Alexander Bertoloni Meli, Teruhisa Koshikawa, Jonathan Leake}

\begin{document}

\maketitle 

\begin{abstract}
In this paper we prove a new characterization of the distinguished unipotent orbits of a connected reductive group over an algebraically closed field of characteristic $0$. For classical groups we prove the characterization by a combinatorial computation, and for exceptional groups we check it with a computer. This characterization is needed in the theory of cuspidal sheaves on the stack of $L$-parameters in the work \cite{BMK} of the first two named authors. 
\end{abstract}

\tableofcontents

\section{Introduction}
Let $G$ be a connected reductive group over an algebraically closed field $k$ of characteristic $0$. Let $\Phi(G,T)$ denote the roots of $G$ relative to a fixed maximal torus $T$. We let $W$ denote the Weyl group of the corresponding root system. We fix a set of simple roots $\Delta \subset \Phi(G,T)$. This determines a set of positive roots $\Phi^+(G,T)$ and we define an absolute value function
\begin{equation*}
    | \cdot | : \Phi(G,T) \rightarrow \Phi^+(G,T),
\end{equation*}
such that $|\alpha|$ is a positive root and equal to $\alpha$ or $-\alpha$.

Now assign a weight function
\begin{equation*}
    \rho: \Delta \rightarrow \Q_{\geq 0}.
\end{equation*}
Since each root in $\Phi(G,T)$ can be expressed uniquely as a linear combination of elements of $\Delta$, this determines a unique extension of $\rho$ to $\Phi(G,T)$. We let $V_i$ denote the set of roots of weight $i$. Weight functions with values in $\{0,1,2\}$ occur naturally in the Bala--Carter classification of unipotent orbits of $G$.

More generally, given a non-archimedean local field $F$, one can attach such a weight function to a unipotent Vogan variety. Namely, we say a semi-simple $L$-parameter $\lambda: W_F \rightarrow G$ is unipotent if it is trivial when restricted to the inertia subgroup $I_F$. Then the Vogan variety $V_{\lambda}$ parametrizing Frobenius semi-simple Weil-Deligne parameters with semisimplification $\lambda$ is given (after possibly passing to a reductive subgroup of $G$ \cite[Theorem 5.1.(d)]{CFMMX}) by the weight $n$ space of some cocharacter $\iota$. Up to passing to rational cocharacters, we may re-normalize so that $n=2$. This weight space is completely determined by the weights of $\iota$ on simple root vectors, and these can be chosen to be non-negative after conjugation, and hence correspond to a weight function.

Conversely, given a weight function $\rho$, let $\nu_{\rho} = \sum\limits_{\alpha \in \Delta} \rho(\alpha) \delta_{\alpha}$. Then define unipotent $\lambda: W_F \rightarrow G$ by $\lambda(\frob) = \nu_{\rho}(q^{1/2})$. Then we can form $V_{\lambda}$ and $H_{\lambda} = Z_G(\lambda)$ in the notation of \cite{BMK}, and the $T$-weights of $V_{\lambda}$ are $V_2$ and the non-zero $T$-weights of $\Lie H_{\lambda}$ are $V_0$.

\begin{definition}{\label{def: scoredef}}
For $w \in W$, we define:
\begin{equation*}
    \zeta(w) := \sum\limits_{v \in V_2} |w(v)| - \sum\limits_{v \in V_0} |w(v)|.
\end{equation*}

Further, we let $\zeta(w)_{\Delta}$ denote the coordinate vector of $\zeta(w)$ in the $\Delta$-basis and let $\zeta(w)_{\alpha}$ denote the coordinate of $\zeta(w)$ with respect to  $\alpha \in \Delta$.

We say that a weight function is \emph{non-degenerate} if for each $w \in W$, the vector $\zeta(w)_{\Delta}$ is strictly positive. \footnote{The inspiration for this terminology is Lemma \ref{lem: fulldimlemma} which proves that $\rho$ is non-degenerate if and only if an associated polytope has full dimension.}
\end{definition}

Recall that (\cite[Proposition 3.2]{BalaCarter}) for any weight function valued in $\{0,1,2\}$ we always have an inequality of cardinalities:
\begin{equation*}
    \# V_2 - \# V_0 \leq \# \Delta.
\end{equation*}
For a general $\rho: \Delta \rightarrow \Q_{\geq 0}$, we say $\rho$ is \emph{distinguished} when $\# V_2 = \# V_0 + \# \Delta$. For a fixed $\Phi(G,T)$ and $\Delta$, the set of distinguished $\rho$ is in bijection with distinguished unipotent orbits of $G$ in the sense of Bala--Carter.

We are now ready to state the main result of the paper.
\begin{theorem}{\label{thm: main}}
The weight function $\rho$ is distinguished if and only if it is non-degenerate.   
\end{theorem}
In particular, this theorem gives a new characterization of the distinguished unipotent orbits of $G$. 

We now give a geometric interpretation of Theorem \ref{thm: main}. We define a certain polytope $\bP_{\cusp} \subset X^{\ast}(T_{\der})_{\R}$ cut out by two affine translates of each hyperplane determined by $w(\delta_{\alpha})$ for $\alpha \in \Delta$ and $w \in W$ and where $\delta_{\alpha}$ is the fundamental co-weight for $\alpha$. These translates have the property that $\zeta(w^{-1})_{\alpha} / ||\delta_{\alpha}||$ is exactly the distance between them. Hence,
\begin{corollary} \label{cor:distinguished-interior}
    We have $\bP_{\cusp}$ is of maximal dimension (equivalently, has non-empty interior inside $X^{\ast}(T_{\der})_{\R}$) if and only if $\rho$ is distinguished.

    The $V_{\lambda}$ attached to $\rho$ as above contains an essentially discrete $L$-parameter if and only if $\rho$ is distinguished \cite[Lemma 9.1.1]{BMK}. Hence, such a $V_{\lambda}$ contains an essentially discrete $L$-parameter if and only if $\bP_{\cusp}$ has nonempty interior.
\end{corollary}
Finally, we have the following simple description of $\bP_{\cusp}$.
\begin{proposition}
    The polytope $\bP_{\cusp}$ is given by 
    \begin{equation*}
        (\bigoplus_{\alpha \in V_2} [0,\alpha]) \ominus (\bigoplus_{\beta \in V_0} [0,\beta]),
    \end{equation*}
    where $\oplus$ is Minkowski sum and $\ominus$ is Minkowski difference and $[0, \alpha]$ is the line segment in $X^{\ast}(T_{\der})_{\R}$ connecting $0$ and $\alpha$. \footnote{Note that the formula in \cite[\S 15.1]{BMK} appears to use a different sign convention, and the discrepancy arises from the anti-dominant normalization of $\iota$ there.}
\end{proposition}
\begin{remark}
    We were inspired to this description of $\bP_{\cusp}$ by similar formulas appearing in a somewhat different context in \cite{PTDolbeault}.
\end{remark}

We conclude with a few preliminary reductions towards the proof of Theorem \ref{thm: main}. We have the following lemma slightly generalizing the standard fact that distinguished unipotent orbits are even (\cite[Theorem 8.2.3]{CollingwoodMcGovern}).
\begin{lemma}
    If $\rho: \Delta \rightarrow \Q_{\geq 0}$ is distinguished, then $\rho$ is valued in $\{0, 2\}$.
\end{lemma}
\begin{proof}
 Attach $\lambda$ to $\rho$ as in the paragraph before Definition \ref{def: scoredef}. Then let $N$ be a representative of the open orbit of $V_{\lambda}$ (which exists by \cite[Proposition 7.3.5]{BMK}). Then $(\lambda, N)$ gives a Weil-Deligne parameter with centralizer $Z_{H_{\lambda}}(N)$ satisfying,
    \begin{equation*}
        \dim Z_{H_{\lambda}}(N) = \dim H_{\lambda} - \dim V_{\lambda} = \dim Z(G).
    \end{equation*}
    Hence $(\lambda, N)$ is essentially discrete (\cite[Definition 7.3.7]{BMK}). Then by \cite[Example 7.1.17]{BMK}, $\rho$ is the weighted Dynkin diagram of a distinguished element $N$, and hence even by \cite[Theorem 8.2.3]{CollingwoodMcGovern}.
\end{proof}
It follows that to prove Theorem \ref{thm: main}, we need to show that $\rho$ non-degenerate implies $\rho$ distinguished and $\rho$ distinguished and valued in $\{0,2\}$ implies $\rho$ non-degenerate.

We remark that the statement of Theorem \ref{thm: main} is insensitive to passing to derived subgroups and central extensions and behaves well with respect to products. Hence we are free to assume that $G$ is simple. We also record the following lemma dealing with the regular unipotent orbit.
\begin{lemma}{\label{lem: regular case}}
    Suppose that $\rho$ assigns weight $2$ to each simple root. Then $\rho$ is non-degenerate.
\end{lemma}
\begin{proof}
    In this case, $V_0 = \varnothing$, so we have $\zeta(w)_{\Delta}$ has non-negative coordinates for each $w \in W$. Positivity then follows from the fact that $w(\Delta)$ is a basis.
\end{proof}

We prove Theorem \ref{thm: main} by analysis of each case, relying heavily on the classification of \cite{BalaCarter}. The $A$-type case is handled by a combinatorial argument and the $B$-, $C$-, and $D$- types are proven using the standard representations of $\Sp_n$ and $\SO_n$ to reduce to the result in $A$-type. It is conceivable to us that one can make a similar reduction for the exceptional groups using their quasi-minuscule representations, but this seems to involve substantial casework. Thus we have opted to check these cases with a computer.

Finally, we briefly remark on our motivation for proving Theorem \ref{thm: main}. In \cite{BMK}, the first two named authors construct a category of ``cuspidal sheaves'' on the stack of $L$-parameters. This category is generated by ``cuspidal vector bundles'' of Vogan varieties and these admit a classification analogous to the cuspidal local systems of Lusztig \cite{LusztigIntersectionCohomologyComplexes}. In working out this classification, we needed to compute the interior points of $\bP_{\cusp}$, and this led us to Theorem \ref{thm: main}.

\begin{remark}
    After work on this document was completed, OpenAI's Astra generated a different and less computational proof of Theorem \ref{thm: main}. This proof (with edits made by the authors of this document and by Anthropic's Claude Opus 5.5) will appear on HexagonMath.
\end{remark}

\section*{Acknowledgements}
A.B.M. would like to thank Phil Engel and David Schwein for their encouragement and helpful conversations on this problem. J.L. would like to thank the math faculty computing facility at the University of Waterloo for the use of their research Linux servers.

A.B.M. was supported by the European Research Council (ERC)
under the European Union’s Horizon 2020 research and innovation programme (grant agreement no. 950326) for part of the completion of this work. T.K. was supported by JSPS KAKENHI Grant Number 24K16895. J.L. acknowledges the support of the Natural Sciences and Engineering Research Council of Canada (NSERC), [funding reference number RGPIN-2023-03726]. Cette recherche a \'et\'e partiellement financ\'ee par le Conseil de recherches en sciences naturelles et en g\'enie du Canada (CRSNG), [num\'ero de r\'ef\'erence RGPIN-2023-03726].

Anthropic's Claude Opus 5.5 (via the Claude team plan for scientists) and OpenAI's Astra were used to proofread the document.

\section{Polytopes}

The condition in Definition \ref{def: scoredef} arises naturally in relation to a certain polytope. Fix a weight function $\rho: \Delta \rightarrow \Q_{\geq 0}$. Let $T_{\der}$ be a maximal torus of $G_{\der}$ and consider $X \coloneq X^{\ast}(T_{\der})_{\R}$. For each simple root $\alpha \in \Delta$, we consider the fundamental co-weight $\delta_{\alpha} \in X_{\ast}(T_{\der})_{\Q}$. For $\nu \in X_{\ast}(T_{\der})_{\Q}$, we use the notation $\Phi^{> 0}_{\nu}$ to denote the subset of $\beta \in \Phi(G,T)$ such that $\langle \nu,  \beta \rangle > 0$. 

\begin{definition}{\label{def: naivepoly}}
  Define the \emph{naive polytope}, $\bP_{\naive}$ to be the set of $\mu \in X$ such that for all $w \in W$ and all $\alpha \in \Delta$
  \begin{equation*}
     \sum\limits_{v \in V_2 \cap \Phi^{<0}_{w(\delta_{\alpha})}} \langle w(\delta_{\alpha}), v \rangle \leq \langle w(\delta_{\alpha}), \mu \rangle \leq \sum\limits_{v \in V_2 \cap \Phi^{>0}_{w(\delta_{\alpha})}} \langle w(\delta_{\alpha}), v \rangle.
  \end{equation*}
Define the \emph{cuspidal polytope}, $\bP_{\cusp}$ to be the set of $\mu \in X$ such that for all $w \in W$ and $\alpha \in \Delta$

\begin{adjustbox}{width = \textwidth, center}

     $\sum\limits_{v \in V_2 \cap \Phi^{<0}_{w(\delta_{\alpha})}} \langle w(\delta_{\alpha}), v \rangle + \sum\limits_{u \in V_0 \cap \Phi^{>0}_{w(\delta_{\alpha})}} \langle w(\delta_{\alpha}), u \rangle  \leq \langle w(\delta_{\alpha}), \mu \rangle \leq \sum\limits_{v \in V_2 \cap \Phi^{>0}_{w(\delta_{\alpha})}} \langle w(\delta_{\alpha}), v \rangle - \sum\limits_{u \in V_0 \cap \Phi^{>0}_{w(\delta_{\alpha})}} \langle w(\delta_{\alpha}), u \rangle.$

\end{adjustbox}

In each case, there is an open polytope denoted $\bP^{\circ}_{\cusp}$  and $\bP^{\circ}_{\naive}$ given by making the inequalities strict.
\end{definition}
\begin{remark}
    Note that these conditions specify that $\mu$ is bounded by affine hyperplanes corresponding to translations of the hyperplanes determined by the $w(\delta_{\alpha})$. Since the $\delta_\alpha$ are spanning, this implies $\bP_{\cusp}$ and $\bP_{\naive}$ are in fact convex polytopes and $\bP_{\cusp} \subset \bP_{\naive}$. In the work \cite{BMK} of the first two named authors, $\bP_{\cusp}$ appears naturally in the context of the Galois side of the categorical Langlands conjecture.
\end{remark}
\begin{lemma}{\label{lem: fulldimlemma}}
    The (signed) distance between the two translations of the hyperplane of $X$ determined by $w(\delta_{\alpha})$ is exactly $\zeta(w^{-1})_{\alpha}/ ||\delta_{\alpha}||$. In particular, $\rho$ is non-degenerate (i.e., $\zeta(w)_{\Delta}$ is strictly positive for all $w \in W$) if and only if $\bP^{\circ}_{\cusp} \neq \emptyset$.
\end{lemma}
\begin{proof}
    For the first claim, it suffices to observe that for all $\mu \in X$, we have
    \begin{equation*}
        \langle w(\delta_{\alpha}), \mu \rangle =\langle \delta_{\alpha} , w^{-1}(\mu)\rangle = w^{-1}(\mu)_{\alpha},
    \end{equation*}

    and that 

    \begin{equation*}
        \zeta(w) = [\sum\limits_{v \in w(V_2) \cap \Phi^+} v - \sum\limits_{u \in w(V_0) \cap \Phi^+} u] - [\sum\limits_{v \in w(V_2) \cap \Phi^-} v + \sum\limits_{u \in w(V_0) \cap \Phi^+} u].
    \end{equation*}
    This implies that if $\bP^{\circ}_{\cusp} \neq \emptyset$, then $\rho$ is non-degenerate.

    To prove the other direction of the second claim, we note that the point $\frac{1}{2}\sum\limits_{v \in V_2} v$ lies midway between the hyperplanes for every pair $(w, \alpha)$ (for instance, see the discussion in \cite[\S15.1]{BMK}). Hence, if $\rho$ is non-degenerate, then $\bP^{\circ}_{\cusp} \neq \emptyset$.
\end{proof}

\begin{definition}
    The \emph{Minkowski sum} $A \oplus B$ of two sets $A,B$ in a Euclidean space $X$ is given by
    \[
        A \oplus B = \{a+b : a \in A, b \in B\},
    \]
    and the \emph{Minkowski difference} $A \ominus B$ is given by one of the following equivalent definitions:
    \begin{enumerate}
        \item $A \ominus B = \{x \in X : \{x\} \oplus B \subseteq A\}$,
        \item $A \ominus B = \left(A^c \oplus (-B)\right)^c$ where $A^c = X \setminus A$ and $-B = \{-b : b \in B\}$,
        \item $A \ominus B = \bigcap_{b \in B} (A \oplus \{-b\})$.
    \end{enumerate}
\end{definition}

\begin{proposition} \label{prop:polytopes}
       The polytope $\bP_{\cusp}$ is given by 
    \begin{equation*}
        (\bigoplus_{\alpha \in V_2} [0,\alpha]) \ominus (\bigoplus_{\beta \in V_0} [0,\beta]),
    \end{equation*}
    and $\bP_{\naive}$ is given by
\begin{equation*}
        \bigoplus_{\alpha \in V_2} [0,\alpha].
\end{equation*}
    Note that $\bP_{\naive} = \{0\}$ when $V_2 = \varnothing$, and $\bP_{\cusp} = \{0\}$ when $V_0 = V_2 = \varnothing$, and $\bP_{\cusp} = \varnothing$ when $V_2 = \varnothing$ and $V_0 \neq \varnothing$.
\end{proposition}
\begin{proof}    
We first prove the result for $\bP_{\naive}$. If $\mu \in   \bigoplus_{\alpha \in V_2} [0,\alpha]$, then $\mu = \sum\limits_{v \in V_2} a_vv$ where $a_v \in [0,1]$. For any fixed $\nu \in X_{\ast}(T_{\der})_{\Q}$, the quantity $\langle \nu, \mu \rangle$ is maximized when $a_v=0$ for each $v \in V_2 \cap \Phi^{<0}_{\nu}$ and $a_v=1$ for $v \in V_2 \cap \Phi^{>0}_{\nu}$ and minimized when $a_v=1$ for each $v \in V_2 \cap \Phi^{<0}_{\nu}$ and $a_v=0$ for $v \in V_2 \cap \Phi^{>0}_{\nu}$. Since $\bigoplus_{\alpha \in V_2} [0,\alpha]$ is convex, this implies
\[
    \bigoplus_{\alpha \in V_2} [0,\alpha] = \left\{\mu : \sum\limits_{v \in V_2 \cap \Phi^{<0}_\nu} \langle \nu, v \rangle \leq \langle \nu, \mu \rangle \leq \sum\limits_{v \in V_2 \cap \Phi^{>0}_\nu} \langle \nu, v \rangle, ~ \forall \nu \in X_{\ast}(T_{\der})\right\}.
\]
It follows that $\bigoplus_{\alpha \in V_2} [0,\alpha] \subseteq \bP_{\naive}$. 

It remains to show that $\bP_{\naive} \subset \bigoplus_{\alpha \in V_2} [0,\alpha]$. We first reduce to the case where $\bigoplus_{\alpha \in V_2} [0,\alpha]$ is the same dimension as the ambient vector space. Let $Y = \Span(V_2) \subset X$. We aim to show $\bP_{\naive} \subset Y$.

We claim $Y \cap \Phi(G,T)$ gives the root system of a Levi subgroup. Indeed, we can choose some finite subset $C \subset X_{\ast}(T_{\der})$ such that $Y$ is the intersection of the kernel of the elements of $C$, thought of as functionals on $X$. Then $Y \cap \Phi(G,T)$ consists of the roots pairing trivially with each element of $C$ and these determine a Levi subgroup $M$ of $G$ (recall that the Levi subgroups of $G$ are precisely the centralizers of tori in $G$ and the image of a cocharacter is a torus \cite[Proposition 1.22]{DigneMichelFiniteGroupsLieType}, and also that $Y$ is spanned by roots of $G$). Then we choose a basis $\Delta'$ of simple roots of $G$ such that $Y$ is spanned by a subset $\Delta'_M \subset \Delta' = w(\Delta)$ and let $Y^{\perp} \subset X_{\ast}(T_{\der})_{\R}$ be the span of cocharacters vanishing on $Y$. We have $Y^{\perp}$ is spanned by $w(\delta_{\alpha})$  such that $w(\alpha) \in \Delta' \setminus \Delta'_M$. Now, the bounds in Definition \ref{def: naivepoly} are $0$ for such $\nu$ and so $\bP_{\naive} \subset Y$.

Now to show  $\bP_{\naive} \subset \bigoplus_{\alpha \in V_2} [0,\alpha]$ it suffices to check that the maximal facets of $\bigoplus_{\alpha \in V_2} [0,\alpha]$ are cut out by affine translates of hyperplanes generated by the span of a subset of some collection of simple roots in $M$ (since these are the hyperplanes cut out by the $w(\delta_{\alpha})$). By \cite[pg 92]{McMullen}, the maximal facets are parallel to hyperplanes through the origin spanned by some subset of $V_2$. By the argument of the previous paragraph, the intersection $H \cap \Phi(M,T)$ gives the root system of a maximal Levi subgroup $L$ of $M$. Let $\Delta''_M$ be a collection of simple roots of $M$ such that the Levi subsystem corresponding to $L$ is given by omitting a single simple root $\alpha''$ from $\Delta''_M$. Extend $\Delta''_M$ to a collection of simple roots $\Delta''$ for $G$  and choose $w$ taking $\Delta$ to $\Delta''$ and $\alpha \in \Delta$ such that $w(\alpha)=\alpha''$. Then $w(\delta_{\alpha}) = \delta_{\alpha''}$ and $H$ is cut out of $Y$ by $w(\delta_{\alpha})$ as desired.

We now consider $\bP_{\cusp}$. For a fixed $\nu \in W \cdot \{\delta_{\alpha}\}_{\alpha \in \Delta}$, let $V^+$ denote the sum of $v \in V_2$ that pair positively with $\nu$ and $R^+$  denote the sum of $r \in V_0$ that pair positively with $\nu$ (and define $V^-, R^-$ analogously). Then $\bP_{\cusp}$ is the set of $\mu \in X$ satisfying 
\begin{equation*}
    \langle \nu, V^- +R^+ \rangle \leq \langle \nu, \mu \rangle \leq \langle \nu, V^+ - R^+ \rangle.
\end{equation*}

We claim this latter set is equivalent to $(\bigoplus_{\alpha \in V_2} [0,\alpha]) \ominus (\bigoplus_{\beta \in V_0} [0,\beta])$. Fix $\nu$ and suppose $x$ is in the Minkowski difference. Then we need to show $\langle \nu, x \rangle$ satisfies the correct bounds. But then $x + R^+$ is in $\bP_{\naive}$ so $\langle \nu, x + R^+ \rangle \leq \langle \nu, V^+ \rangle$, which implies $\langle \nu, x \rangle \leq \langle \nu, V^+ \rangle - \langle \nu, R^+ \rangle$. Similarly, $x - R^+$ is in the naive polytope (since $V_0$ contains negatives of all its weights) so $\langle \nu, V^- \rangle + \langle \nu, R^+ \rangle \leq \langle \nu, x \rangle$. Conversely, if $x$ satisfies these bounds for all $\nu$, then we can prove by the first part, that $x+r \in \bP_{\naive} = \bigoplus_{\alpha \in V_2} [0,\alpha]$ for each $r \in \bigoplus_{\beta \in V_0} [0,\beta]$, and this implies that $x$ is in the Minkowski difference. 
\end{proof}

\section{\texorpdfstring{$A_n$}{An} case}
In the $A_n$ case, the only distinguished $\rho$ is the regular one which takes each simple root to $2$. Lemma \ref{lem: regular case} proves that this is non-degenerate. It remains to prove the converse.

We first consider the case where $\rho$ is valued in $\{0,2\}$ and compute the $\zeta(w)_{\alpha}$. We work with the standard presentation of the $A_n$ root system where $\Delta=\{ \bm{e}_1-\bm{e}_2, ..., \bm{e}_{n}-\bm{e}_{n+1}\}$ and order the simple roots as indicated. We often use the notation $\alpha_i$  for  $\bm{e}_i - \bm{e}_{i+1}$.

In order to facilitate computation, we depict the sets $V_2$ and $V_0$ as edges of a certain graph. The vertices of the graph are the numbers $1, ..., n+1$ and the root $\bm{e}_i - \bm{e}_j$ (and its negative) will be encoded as a single edge between $i$ and $j$. To elucidate the structure of the graph corresponding to $V_2$,  we group the vertices as follows. Suppose that the weight $2$ simple roots are $e_{i_j} - e_{i_j +1}$ for $0 \leq j \leq m-1$. Then we define $H_0=\{1,...,i_0\}$ and $H_{j+1}=\{ i_j + 1,  \dots, i_{j+1}\}$ for $1 \leq j \leq m-2$ and $H_m = \{ i_{m-1}+1, \dots, n+1\}$ where $m$ is the cardinality $\# \rho^{-1}(2)$. We denote $N_i = \# H_i$.

Observe now that $V_2$ corresponds to the complete bipartite graph between adjacent groupings. In other words, we include all edges between $H_0$ and $H_1$, between $H_1$ and $H_2$, and so on. The set $V_0$ corresponds to taking the complete graphs of each of the $H_i$, though we note that we need to consider each such edge with multiplicity $2$ since a root has weight $0$ if and only if its negative does.

With this graphical interpretation, we have a simple way of computing $\zeta(w)_{\alpha_i}$. When $w=1$, we simply count the number of $V_2$ edges between $\{1, ..., i\}$ and $\{i+1, ..., n+1\}$ and subtract twice the number of $V_0$ edges between these vertices. For $w$ general, the idea is the same but the two groups of vertices are $w^{-1}(\{1, ..., i\})$ and $w^{-1}(\{i+1, ..., n+1\})$. Hence proving the positivity of $\zeta(w)_{\Delta}$ for each $w \in W$ is equivalent to showing that for each splitting of $\{1, ..., n+1\}$ into two groups, the number of $V_2$ edges between the groups is larger than twice the number of $V_0$ edges.

Suppose now that $\rho$ gives us a partition of $\{1, ..., n+1\}$ into $H_0, ..., H_m$ (we may assume $m \geq 1$ as the $m=0$ case is degenerate) and that for some arbitrary splitting of $\{1, ..., n+1\}$ into two groups, we have that $H_j$ has $a_j$ vertices in the first group and hence $N_j-a_j$ vertices in the second group. Then for any pair $(w,i)$ corresponding to this splitting, we have
\begin{equation*}
    \zeta(w)_{\alpha_i} = [\sum\limits^{m}_{j=1} a_{j-1}(N_j-a_j) + a_j(N_{j-1}-a_{j-1})]  - 2[\sum\limits^{m}_{j=0} a_j(N_j-a_j)].
\end{equation*}
We rewrite this as 
\begin{equation*}
  \begin{split}
    \zeta(w)_{\alpha_i} &= \sum_{j=1}^{m}
        \begin{bmatrix}
            a_{j-1} \\ N_{j-1}-a_{j-1}
        \end{bmatrix}^\top
        \begin{bmatrix}
            0 & 1  \\
            1 & 0 
        \end{bmatrix}
        \begin{bmatrix}
            a_j \\ N_j-a_j
        \end{bmatrix}
        \\
        &- \sum_{j=0}^{m} 
        \begin{bmatrix}
            a_j \\ N_j-a_j
        \end{bmatrix}^\top
        \begin{bmatrix}
            0 & 1 \\
            1 & 0 
        \end{bmatrix}
        \begin{bmatrix}
            a_j \\ N_j-a_j
        \end{bmatrix}. \\
\end{split}
\end{equation*}
Another form of this expression is given by formally defining $N_{-1}=a_{-1} = N_{m+1}=a_{m+1}=0$ and writing $\bm{v}_j = [a_j, N_j -a_j]^\top$ so that
\begin{equation*}
    \zeta(w)_{\alpha_i}=-\frac{1}{2} \sum_{j=0}^{m+1} (\bm{v}_j-\bm{v}_{j-1})^\top
        \begin{bmatrix}
            0 & 1 \\
            1 & 0 
        \end{bmatrix}
        (\bm{v}_j-\bm{v}_{j-1}).
\end{equation*}

Now consider what happens if we have, for $1 \leq i \leq m$ fixed,
\begin{equation*}
    a_j = \begin{cases}
    0 &  j=-1, m+1\\
    1 & -1 < j < i\\
    0 & i  \leq j < m+1.
\end{cases}
\end{equation*}
Then
\begin{equation*}
   \zeta(w)_{\alpha_i} = (N_i - N_{i-1}+1)-(N_0-1).
\end{equation*}
Hence, $\zeta(w)_{\alpha_i} >0$ implies $N_i \geq N_{i-1} + (N_0-1)$. Since $N_0 -1 \geq 0$, we have $N_i \geq N_{i-1}$ with equality only if $N_0=1$. If we instead stipulate that the $a_j$ for $i \leq j < m+1$ are $1$ and the other $a_j$ are $0$, we deduce analogously (by evaluating $\zeta(w)_{\alpha_{m+1-i}}$) that $N_{i-1} \geq N_i + (N_m -1)$. Hence if $\zeta(w)_{\alpha_i}>0$ for all $i$, we must have that all $N_i=1$ for $0 \leq i \leq m$. This completes the proof of Theorem \ref{thm: main} in the $A_n$ case under our assumption on $\rho$.

We now take a general weight function $\rho$ valued in $\Q_{\geq 0}$ and reduce to the previous case. As before, we can define a graph whose vertices are $1, \dots, n+1$ with groupings $H_i$ determined by the locations of nonzero $\rho(\alpha_i)$. We argue that there exists some $w \in W$ and a weight function $\rho'$ valued in $\{0, 2, 3\}$ such that the sets $V'_2, V'_0$ relative to $\rho'$ satisfy $V'_2 = w(V_2)$ and $V'_0=w(V_0)$, and moreover $V'_2$ contains only positive roots. Consider the grouping $H_0$. There exists at most one other grouping $H_i$ such that the vertices in $H_0$ are connected to vertices of $H_i$ by edges. Similarly, the vertices of $H_i$ are connected to $H_0$ by edges and at most one other grouping $H_j$ which must satisfy $j>i$. More generally, for any $k$ the vertices of $H_k$ are connected to the vertices of at most two other groupings: at most one grouping $H_i$ for which $i < k$ and at most one grouping $H_j$ for which $k < j$. In this way, the graph of $V_2$ splits into some number of connected components. Choose $w$ that permutes the groupings (while preserving the ordering of the vertices within each grouping and preserving the ordering of the groupings in a given component) such that groupings in the same connected component are adjacent. Then $w(V_2)$ and $w(V_0)$ are determined by a weight function $\rho'$ that takes value $0$ on each simple root in $w(V_0)$, takes value $2$ on each simple root separating groupings within the same connected component, and takes value $3$ on each simple root separating non-connected groupings.

We observe that if $\rho$ takes a value other than $0$ or $2$ at some $\alpha_i$, then the groupings form at least $2$ connected components (since the groupings separated by $\alpha_i$ cannot be in the same connected component). Thus, $\rho'$ will take at least one simple root $\alpha_i$ to $3$ and then $0=\zeta_{\rho'}(1)_{\alpha_i}=\zeta_{\rho}(w)_{\alpha_i}$. This implies that any $\rho$ with $\bP^{\circ}_{\cusp} \neq \emptyset$ must be valued in $\{0,2\}$ and completes Theorem \ref{thm: main} in $A$-type.

We now compute some examples to help depict the above descriptions.

\begin{example}
    Let $n=5$ and consider $\rho$ which takes values $\{0,2,0,0,2\}$ on the $5$ simple roots. We denote this on the Dynkin diagram of $A_5$ by:
    \[
    \begin{tikzpicture}[baseline=-2pt,scale=1,
        node/.style={circle,draw,inner sep=2pt},
        edgelabel/.style={midway,above}
    ]
        \node[draw=none,fill=none] (a0) at (0,0) {};
        \node[node] (a1) at (1.5,0) {$0$};
        \node[node] (a2) at (3,0) {$2$};
        \node[node] (a3) at (4.5,0) {$0$};
        \node[node] (a4) at (6,0) {$0$};
        \node[node] (a5) at (7.5,0) {$2$};
        \node[draw=none,fill=none] (a6) at (9,0) {};

        \draw[draw=none] (a0) -- node[edgelabel] {} (a1);
        \draw (a1) -- node[edgelabel] {} (a2);
        \draw (a2) -- node[edgelabel] {} (a3);
        \draw (a3) -- node[edgelabel] {} (a4);
        \draw (a4) -- node[edgelabel] {} (a5);
        \draw[draw=none] (a5) -- node[edgelabel] {} (a6);
    \end{tikzpicture}
    \]
    We then add labels corresponding to the numbers $1,\ldots, n+1$ which we will split into groupings via:
    \[
    \begin{tikzpicture}[baseline=-2pt,scale=1,
        node/.style={circle,draw,inner sep=2pt},
        edgelabel/.style={midway,above}
    ]
        \node[draw=none,fill=none] (a0) at (0,0) {};
        \node[node] (a1) at (1.5,0) {$0$};
        \node[node] (a2) at (3,0) {$2$};
        \node[node] (a3) at (4.5,0) {$0$};
        \node[node] (a4) at (6,0) {$0$};
        \node[node] (a5) at (7.5,0) {$2$};
        \node[draw=none,fill=none] (a6) at (9,0) {};

        \draw[draw=none] (a0) -- node[edgelabel] {$\{1$} (a1);
        \draw (a1) -- node[edgelabel] {$2\}$} (a2);
        \draw (a2) -- node[edgelabel] {$\{3$} (a3);
        \draw (a3) -- node[edgelabel] {$\textcolor{white}{\{}4\textcolor{white}{\}}$} (a4);
        \draw (a4) -- node[edgelabel] {$5\}$} (a5);
        \draw[draw=none] (a5) -- node[edgelabel] {$\{6\}$} (a6);
    \end{tikzpicture}
    \]
    Note in this case there are $3$ groups, given by $\{1,2\}$, $\{3,4,5\}$, and $\{6\}$, delineated precisely by simple roots for which the value of $\rho$ is nonzero. As described above, we consider two types of edges:
    \begin{itemize}
        \item those associated to $V_2$, given by complete bipartite graphs between groupings which are separated by simple roots with $\rho$-weights which sum to $2$, and
        \item those associated to $V_0$, given by complete graphs on the groupings themselves. (Recall that these edges should be considered to have multiplicity $2$, since a root has weight $0$ if and only if its negative does.)
    \end{itemize}
    In this case, these edges are given as follows:
    \[
    \begin{tikzpicture}[baseline=-2pt,scale=1,
        node/.style={circle,draw,inner sep=2pt},
        edgelabel/.style={midway,above}
    ]
        \node[draw=none,fill=none] (a0) at (0,0) {};
        \node[node] (a1) at (1.5,0) {$0$};
        \node[node] (a2) at (3,0) {$2$};
        \node[node] (a3) at (4.5,0) {$0$};
        \node[node] (a4) at (6,0) {$0$};
        \node[node] (a5) at (7.5,0) {$2$};
        \node[draw=none,fill=none] (a6) at (9,0) {};

        \draw[draw=none] (a0) -- node[edgelabel] {$\{1$} (a1);
        \draw (a1) -- node[edgelabel] {$2\}$} (a2);
        \draw (a2) -- node[edgelabel] {$\{3$} (a3);
        \draw (a3) -- node[edgelabel] {$\textcolor{white}{\{}4\textcolor{white}{\}}$} (a4);
        \draw (a4) -- node[edgelabel] {$5\}$} (a5);
        \draw[draw=none] (a5) -- node[edgelabel] {$\{6\}$} (a6);

        \draw[bend left=60,color=red,thick]($(a0)!0.5!(a1) + (0,0.6)$) to ($(a1)!0.5!(a2) + (0,0.6)$);
        \draw[bend left=60]($(a0)!0.5!(a1) + (0,0.6)$) to ($(a2)!0.5!(a3) + (0,0.6)$);
        \draw[bend left=60]($(a0)!0.5!(a1) + (0,0.6)$) to ($(a3)!0.5!(a4) + (0,0.6)$);
        \draw[bend left=60]($(a0)!0.5!(a1) + (0,0.6)$) to ($(a4)!0.5!(a5) + (0,0.6)$);

        \draw[bend left=60]($(a1)!0.5!(a2) + (0,0.6)$) to ($(a2)!0.5!(a3) + (0,0.6)$);
        \draw[bend left=60]($(a1)!0.5!(a2) + (0,0.6)$) to ($(a3)!0.5!(a4) + (0,0.6)$);
        \draw[bend left=60]($(a1)!0.5!(a2) + (0,0.6)$) to ($(a4)!0.5!(a5) + (0,0.6)$);

        \draw[bend left=60,color=red,thick]($(a2)!0.5!(a3) + (0,0.6)$) to ($(a3)!0.5!(a4) + (0,0.6)$);
        \draw[bend left=60,color=red,thick]($(a2)!0.5!(a3) + (0,0.6)$) to ($(a4)!0.5!(a5) + (0,0.6)$);
        \draw[bend left=60]($(a2)!0.5!(a3) + (0,0.6)$) to ($(a5)!0.5!(a6) + (0,0.6)$);

        \draw[bend left=60,color=red,thick]($(a3)!0.5!(a4) + (0,0.6)$) to ($(a4)!0.5!(a5) + (0,0.6)$);
        \draw[bend left=60]($(a3)!0.5!(a4) + (0,0.6)$) to ($(a5)!0.5!(a6) + (0,0.6)$);
        
        \draw[bend left=60]($(a4)!0.5!(a5) + (0,0.6)$) to ($(a5)!0.5!(a6) + (0,0.6)$);
    \end{tikzpicture}
    \]
    That is, between each pair of adjacent groupings (i.e. between $\{1,2\}$ and $\{3,4,5\}$ and between $\{3,4,5\}$ and $\{6\}$) we have complete bipartite graphs (black edges) corresponding to $V_2$ edges, and within each grouping we have complete graphs (red edges) corresponding to $V_0$ edges. We do not want to draw all of these edges for all examples, so we keep a single black edge for each complete bipartite graph between groupings and a single red edge for each complete graph within each grouping. For example, the above diagram would become:
    \[
    \begin{tikzpicture}[baseline=-2pt,scale=1,
        node/.style={circle,draw,inner sep=2pt},
        edgelabel/.style={midway,above}
    ]
        \node[draw=none,fill=none] (a0) at (0,0) {};
        \node[node] (a1) at (1.5,0) {$0$};
        \node[node] (a2) at (3,0) {$2$};
        \node[node] (a3) at (4.5,0) {$0$};
        \node[node] (a4) at (6,0) {$0$};
        \node[node] (a5) at (7.5,0) {$2$};
        \node[draw=none,fill=none] (a6) at (9,0) {};

        \draw[draw=none] (a0) -- node[edgelabel] {$\{1$} (a1);
        \draw (a1) -- node[edgelabel] {$2\}$} (a2);
        \draw (a2) -- node[edgelabel] {$\{3$} (a3);
        \draw (a3) -- node[edgelabel] {$\textcolor{white}{\{}4\textcolor{white}{\}}$} (a4);
        \draw (a4) -- node[edgelabel] {$5\}$} (a5);
        \draw[draw=none] (a5) -- node[edgelabel] {$\{6\}$} (a6);

        \draw[bend left=60,color=red,thick]($(a0)!0.5!(a1) + (0,0.6)$) to ($(a1)!0.5!(a2) + (0,0.6)$);

        \draw[bend left=60]($(a1)!0.5!(a2) + (0,0.6)$) to ($(a2)!0.5!(a3) + (0,0.6)$);

        \draw[bend left=60,color=red,thick]($(a2)!0.5!(a3) + (0,0.6)$) to ($(a4)!0.5!(a5) + (0,0.6)$);

        
        \draw[bend left=60]($(a4)!0.5!(a5) + (0,0.6)$) to ($(a5)!0.5!(a6) + (0,0.6)$);
    \end{tikzpicture}
    \]
    Again, we have one black edge for each adjacent pair of groupings and one red edge for each grouping itself.
\end{example}

\begin{example}
    Let $n=5$ and consider $\rho$ which takes values $\{2,1,1,1,2\}$ on the $5$ simple roots. The weighted Dynkin diagram of $A_5$, with $\rho$ values, $V_2$ black edges, and $V_0$ red edges, is given by:
    \[
    \begin{tikzpicture}[baseline=-2pt,scale=1,
        node/.style={circle,draw,inner sep=2pt},
        edgelabel/.style={midway,above}
    ]
        \node[draw=none,fill=none] (a0) at (0,0) {};
        \node[node] (a1) at (1.5,0) {$2$};
        \node[node] (a2) at (3,0) {$1$};
        \node[node] (a3) at (4.5,0) {$1$};
        \node[node] (a4) at (6,0) {$1$};
        \node[node] (a5) at (7.5,0) {$2$};
        \node[draw=none,fill=none] (a6) at (9,0) {};
    

        \draw[draw=none] (a0) -- node[edgelabel] {$\{1\}$} (a1);
        \draw (a1) -- node[edgelabel] {$\{2\}$} (a2);
        \draw (a2) -- node[edgelabel] {$\{3\}$} (a3);
        \draw (a3) -- node[edgelabel] {$\{4\}$} (a4);
        \draw (a4) -- node[edgelabel] {$\{5\}$} (a5);
        \draw[draw=none] (a5) -- node[edgelabel] {$\{6\}$} (a6);

        \draw[bend left=60]($(a0)!0.5!(a1) + (0,0.6)$) to ($(a1)!0.5!(a2) + (0,0.6)$);
        \draw[bend left=60]($(a1)!0.5!(a2) + (0,0.6)$) to ($(a3)!0.5!(a4) + (0,0.6)$);
        \draw[bend left=60]($(a2)!0.5!(a3) + (0,0.6)$) to ($(a4)!0.5!(a5) + (0,0.6)$);
        \draw[bend left=60]($(a4)!0.5!(a5) + (0,0.6)$) to ($(a5)!0.5!(a6) + (0,0.6)$);
    \end{tikzpicture}
    \]
    Note for our example there are no $V_0$ red edges. Now suppose we apply the Weyl element $w'$ given by swapping $3$ and $4$, giving:
    \[
    \begin{tikzpicture}[baseline=-2pt,scale=1,
        node/.style={circle,draw,inner sep=2pt},
        edgelabel/.style={midway,above}
    ]
        \node[draw=none,fill=none] (a0) at (0,0) {};
        \node[node] (a1) at (1.5,0) {$2$};
        \node[node] (a2) at (3,0) {$1$};
        \node[node] (a3) at (4.5,0) {$1$};
        \node[node] (a4) at (6,0) {$1$};
        \node[node] (a5) at (7.5,0) {$2$};
        \node[draw=none,fill=none] (a6) at (9,0) {};
    

        \draw[draw=none] (a0) -- node[edgelabel] {$\{1\}$} (a1);
        \draw (a1) -- node[edgelabel] {$\{2\}$} (a2);
        \draw (a2) -- node[edgelabel] {$\{4\}$} (a3);
        \draw (a3) -- node[edgelabel] {$\{3\}$} (a4);
        \draw (a4) -- node[edgelabel] {$\{5\}$} (a5);
        \draw[draw=none] (a5) -- node[edgelabel] {$\{6\}$} (a6);

        \draw[bend left=60]($(a0)!0.5!(a1) + (0,0.6)$) to ($(a1)!0.5!(a2) + (0,0.6)$);
        \draw[bend left=60]($(a1)!0.5!(a2) + (0,0.6)$) to ($(a2)!0.5!(a3) + (0,0.6)$);
        \draw[bend left=60]($(a3)!0.5!(a4) + (0,0.6)$) to ($(a4)!0.5!(a5) + (0,0.6)$);
        \draw[bend left=60]($(a4)!0.5!(a5) + (0,0.6)$) to ($(a5)!0.5!(a6) + (0,0.6)$);
    \end{tikzpicture}
    \]
    That is, the original diagram had two connected components with respect to the $V_2$ and $V_0$ edges, and by applying $w'$ we were able to move the two components to the two sides of the diagram. If we consider $\rho'$ which takes values $\{2,2,3,2,2\}$ on the $5$ simple roots, then we obtain the same diagram, implying that the values of $\zeta$ for $\rho$ and for $\rho'$ are equal up to the action of $w'$. The diagram in this case looks like:
    \[
    \begin{tikzpicture}[baseline=-2pt,scale=1,
        node/.style={circle,draw,inner sep=2pt},
        edgelabel/.style={midway,above}
    ]
        \node[draw=none,fill=none] (a0) at (0,0) {};
        \node[node] (a1) at (1.5,0) {$2$};
        \node[node] (a2) at (3,0) {$2$};
        \node[node] (a3) at (4.5,0) {$3$};
        \node[node] (a4) at (6,0) {$2$};
        \node[node] (a5) at (7.5,0) {$2$};
        \node[draw=none,fill=none] (a6) at (9,0) {};
    

        \draw[draw=none] (a0) -- node[edgelabel] {$\{1\}$} (a1);
        \draw (a1) -- node[edgelabel] {$\{2\}$} (a2);
        \draw (a2) -- node[edgelabel] {$\{3\}$} (a3);
        \draw (a3) -- node[edgelabel] {$\{4\}$} (a4);
        \draw (a4) -- node[edgelabel] {$\{5\}$} (a5);
        \draw[draw=none] (a5) -- node[edgelabel] {$\{6\}$} (a6);

        \draw[bend left=60]($(a0)!0.5!(a1) + (0,0.6)$) to ($(a1)!0.5!(a2) + (0,0.6)$);
        \draw[bend left=60]($(a1)!0.5!(a2) + (0,0.6)$) to ($(a2)!0.5!(a3) + (0,0.6)$);
        \draw[bend left=60]($(a3)!0.5!(a4) + (0,0.6)$) to ($(a4)!0.5!(a5) + (0,0.6)$);
        \draw[bend left=60]($(a4)!0.5!(a5) + (0,0.6)$) to ($(a5)!0.5!(a6) + (0,0.6)$);
    \end{tikzpicture}
    \]
    Here, $\rho'$ (and thus $\rho$ also) is degenerate because $\zeta(1)_{\alpha_3} = 0$.
\end{example}

\begin{example}
    Let $n=7$ and consider $\rho$ which takes values $\{\frac{3}{2},0,\frac{1}{2},0,\frac{1}{2},1,\frac{1}{2}\}$ on the $7$ simple roots. The weighted Dynkin diagram of $A_7$, with $\rho$ values, $V_2$ black edges, and $V_0$ red edges, is given by:
    \[
    \begin{tikzpicture}[baseline=-2pt,scale=1,
        node/.style={circle,draw,inner sep=2pt},
        edgelabel/.style={midway,above}
    ]
        \node[draw=none,fill=none] (a0) at (0,0) {};
        \node[node] (a1) at (1.5,0) {\scriptsize$\frac{3}{2}$};
        \node[node] (a2) at (3,0) {$0$};
        \node[node] (a3) at (4.5,0) {\scriptsize$\frac{1}{2}$};
        \node[node] (a4) at (6,0) {$0$};
        \node[node] (a5) at (7.5,0) {\scriptsize$\frac{1}{2}$};
        \node[node] (a6) at (9,0) {$1$};
        \node[node] (a7) at (10.5,0) {\scriptsize$\frac{1}{2}$};
        \node[draw=none,fill=none] (a8) at (12,0) {};

        \draw[draw=none] (a0) -- node[edgelabel] {$\{1\}$} (a1);
        \draw (a1) -- node[edgelabel] {$\{2$} (a2);
        \draw (a2) -- node[edgelabel] {$3\}$} (a3);
        \draw (a3) -- node[edgelabel] {$\{4$} (a4);
        \draw (a4) -- node[edgelabel] {$5\}$} (a5);
        \draw (a5) -- node[edgelabel] {$\{6\}$} (a6);
        \draw (a6) -- node[edgelabel] {$\{7\}$} (a7);
        \draw[draw=none] (a7) -- node[edgelabel] {$\{8\}$} (a8);

        \draw[bend left=60]($(a0)!0.5!(a1) + (0,0.6)$) to ($(a3)!0.5!(a4) + (0,0.6)$);
        \draw[bend left=60]($(a2)!0.5!(a3) + (0,0.6)$) to ($(a6)!0.5!(a7) + (0,0.6)$);
        \draw[bend left=60]($(a4)!0.5!(a5) + (0,0.6)$) to ($(a7)!0.5!(a8) + (0,0.6)$);
        \draw[bend left=60,color=red,thick]($(a1)!0.5!(a2) + (0,0.6)$) to ($(a2)!0.5!(a3) + (0,0.6)$);
        \draw[bend left=60,color=red,thick]($(a3)!0.5!(a4) + (0,0.6)$) to ($(a4)!0.5!(a5) + (0,0.6)$);
    \end{tikzpicture}
    \]
    In this case we have three connected components with respect to the $V_2$ and $V_0$ edges (i.e. $\{1,4,5,8\}$, $\{2,3,7\}$, and $\{6\}$), and thus we apply the Weyl element $w' = \left(\begin{smallmatrix}1&2&3&4&5&6&7&8\\1&6&7&2&3&5&8&4\end{smallmatrix}\right)$, giving:
    \[
    \begin{tikzpicture}[baseline=-2pt,scale=1,
        node/.style={circle,draw,inner sep=2pt},
        edgelabel/.style={midway,above}
    ]
        \node[draw=none,fill=none] (a0) at (0,0) {};
        \node[node] (a1) at (1.5,0) {\scriptsize$\frac{3}{2}$};
        \node[node] (a2) at (3,0) {$0$};
        \node[node] (a3) at (4.5,0) {\scriptsize$\frac{1}{2}$};
        \node[node] (a4) at (6,0) {$0$};
        \node[node] (a5) at (7.5,0) {\scriptsize$\frac{1}{2}$};
        \node[node] (a6) at (9,0) {$1$};
        \node[node] (a7) at (10.5,0) {\scriptsize$\frac{1}{2}$};
        \node[draw=none,fill=none] (a8) at (12,0) {};

        \draw[draw=none] (a0) -- node[edgelabel] {$\{1\}$} (a1);
        \draw (a1) -- node[edgelabel] {$\{4$} (a2);
        \draw (a2) -- node[edgelabel] {$5\}$} (a3);
        \draw (a3) -- node[edgelabel] {$\{8\}$} (a4);
        \draw (a4) -- node[edgelabel] {$\{6\}$} (a5);
        \draw (a5) -- node[edgelabel] {$\{2$} (a6);
        \draw (a6) -- node[edgelabel] {$3\}$} (a7);
        \draw[draw=none] (a7) -- node[edgelabel] {$\{7\}$} (a8);

        \draw[bend left=60]($(a0)!0.5!(a1) + (0,0.6)$) to ($(a1)!0.5!(a2) + (0,0.6)$);
        \draw[bend left=60]($(a2)!0.5!(a3) + (0,0.6)$) to ($(a3)!0.5!(a4) + (0,0.6)$);
        \draw[bend left=60]($(a6)!0.5!(a7) + (0,0.6)$) to ($(a7)!0.5!(a8) + (0,0.6)$);
        \draw[bend left=60,color=red,thick]($(a1)!0.5!(a2) + (0,0.6)$) to ($(a2)!0.5!(a3) + (0,0.6)$);
        \draw[bend left=60,color=red,thick]($(a5)!0.5!(a6) + (0,0.6)$) to ($(a6)!0.5!(a7) + (0,0.6)$);
    \end{tikzpicture}
    \]
    If we consider $\rho'$ which takes values $\{2,0,2,3,3,0,2\}$ on the $7$ simple roots, then we obtain the same diagram, implying that the values of $\zeta$ for $\rho$ and for $\rho'$ are equal up to the action of $w'$. The diagram in this case looks like:
    \[
    \begin{tikzpicture}[baseline=-2pt,scale=1,
        node/.style={circle,draw,inner sep=2pt},
        edgelabel/.style={midway,above}
    ]
        \node[draw=none,fill=none] (a0) at (0,0) {};
        \node[node] (a1) at (1.5,0) {$2$};
        \node[node] (a2) at (3,0) {$0$};
        \node[node] (a3) at (4.5,0) {$2$};
        \node[node] (a4) at (6,0) {$3$};
        \node[node] (a5) at (7.5,0) {$3$};
        \node[node] (a6) at (9,0) {$0$};
        \node[node] (a7) at (10.5,0) {$2$};
        \node[draw=none,fill=none] (a8) at (12,0) {};

        \draw[draw=none] (a0) -- node[edgelabel] {$\{1\}$} (a1);
        \draw (a1) -- node[edgelabel] {$\{2$} (a2);
        \draw (a2) -- node[edgelabel] {$3\}$} (a3);
        \draw (a3) -- node[edgelabel] {$\{4\}$} (a4);
        \draw (a4) -- node[edgelabel] {$\{5\}$} (a5);
        \draw (a5) -- node[edgelabel] {$\{6$} (a6);
        \draw (a6) -- node[edgelabel] {$7\}$} (a7);
        \draw[draw=none] (a7) -- node[edgelabel] {$\{8\}$} (a8);

        \draw[bend left=60]($(a0)!0.5!(a1) + (0,0.6)$) to ($(a1)!0.5!(a2) + (0,0.6)$);
        \draw[bend left=60]($(a2)!0.5!(a3) + (0,0.6)$) to ($(a3)!0.5!(a4) + (0,0.6)$);
        \draw[bend left=60]($(a6)!0.5!(a7) + (0,0.6)$) to ($(a7)!0.5!(a8) + (0,0.6)$);
        \draw[bend left=60,color=red,thick]($(a1)!0.5!(a2) + (0,0.6)$) to ($(a2)!0.5!(a3) + (0,0.6)$);
        \draw[bend left=60,color=red,thick]($(a5)!0.5!(a6) + (0,0.6)$) to ($(a6)!0.5!(a7) + (0,0.6)$);
    \end{tikzpicture}
    \]
    Here, $\rho'$ (and thus $\rho$ also) is degenerate because $\zeta(1)_{\alpha_4} = \zeta(1)_{\alpha_5} = 0$.
\end{example}

\section{Reduction to \texorpdfstring{$A$}{A}-type for classical groups} \label{sec:A-reduction}

The standard embeddings of the groups $\SO_{2n+1}, \Sp_{2n}, \SO_{2n}$ respectively into $\GL_{2n+1}, \GL_{2n}, \GL_{2n}$ (normalized such that the intersection of the diagonal torus with the image of the embedding is a maximal torus of the respective group) will allow us to prove the $B, C, D$ cases of Theorem \ref{thm: main}.

These embeddings $H \hookrightarrow G$ do not always induce maps $\Phi(G,T) \rightarrow \Phi(H,T_H)$ (for $T_H = T \cap H$), but they do induce maps of the respective root lattices.

Given such a map $f: R \to S$ between root lattices we define $f^{-1}(s)$ to be the set of \emph{roots} $r \in R$ which map to a given root $s \in S$ under $f$. We call such a map \textbf{positive} if $f^{-1}(s)$ consists of only positive roots for all positive roots $s \in S$ and \textbf{root surjective} if $f^{-1}(s)$ is non-empty for every root $s \in S$.

\begin{lemma}
    If $f: R \to S$ is positive then for every root $s \in S$ we have $f^{-1}(|s|) = |f^{-1}(s)|$, where $|X| = \{|x|: x \in X\}$.
\end{lemma}
\begin{proof}
    Let $s \in S$ be a root. If $s$ is positive then $f^{-1}(|s|) = f^{-1}(s) = |f^{-1}(s)|$ by positivity of $f$. If $s$ is negative then $-s$ is positive, which implies $f^{-1}(|s|) = f^{-1}(-s) = |f^{-1}(-s)| = |-f^{-1}(s)| = |f^{-1}(s)|$.
\end{proof}

Given a weighting function $\rho$ of $\Delta$, we abuse notation as before and let $\rho$ also refer to the function on the root lattice given by linear extension. Given a positive $f$ and a weighting $\rho$ of the simple roots of $S$, we construct a (not necessarily non-negative) weighting $f^{-1}(\rho)$ of the simple roots of $R$ by defining $f^{-1}(\rho) = \rho \circ f$.

\begin{lemma} \label{lem:k-roots}
    Fix $f: R \to S$ and let $\rho$ be a weighting of the simple roots of $S$, with associated weighting $f^{-1}(\rho)$ of the simple roots of $R$. Then $f^{-1}(V_k(S)) \subseteq V_k(R)$ for all $k$.
\end{lemma}
\begin{proof}
    Let $s \in S$ be a root, and let $r \in R$ be a root such that $f(r) = s$. Thus $r = \sum_i \xi_i \alpha_i$ where $\alpha_i$ are simple roots and $\xi_i \in \Z$. We compute
    \[
        f^{-1}(\rho)[r] = \sum_i \xi_i f^{-1}(\rho)[\alpha_i] = \sum_i \xi_i \rho[f(\alpha_i)] = \rho\left[\sum_i \xi_i f(\alpha_i)\right] = \rho[s].
    \]
    Therefore if $s \in V_k(S)$ then $r \in V_k(R)$.
\end{proof}

Let $\phi: W'_S \to W'_R$ be a function between groups of root-preserving automorphisms of $S$ and $R$ respectively. In all cases $W'_R$ will be the Weyl group $W_R$. In the $B$- and $C$-type cases $W'_S$ will be the Weyl group $W_S$. In the $D$-type cases, $W'_S$ will be the degree $2$ extension of $W_S$ generated by the nontrivial automorphism of the Dynkin diagram. Note also that for $w \in W'_S$, we can define $\zeta(w)_{\Delta_S}$ in analogy with the $W_S$ case. For the standard embeddings of classical groups, $\phi$ will be the embedding given by considering signed permutations. The function $\phi$ is said to be \textbf{compatible with $f$} if for any root $s \in S$ and any $w \in W'_S$ we have
\[
    f^{-1}(w \cdot s) = \phi(w) \cdot f^{-1}(s).
\]

\begin{theorem} \label{thm:score}
    Let $f: R \to S$ be positive and root surjective, and let $\phi: W'_S \to W'_R$ be compatible with $f$. Then for all $w \in W'_S$ we have 
    \[
        \zeta_S(w) = f\left[\sum\limits_{r \in f^{-1}(V_2(S))} \frac{|\phi(w) \cdot r|}{m_{|w \cdot f(r)|}} - \sum\limits_{r \in f^{-1}(V_0(S))} \frac{|\phi(w) \cdot r|}{m_{|w \cdot f(r)|}}\right],
    \]
    where $m_s$ denotes the number of roots in the fiber $f^{-1}(s)$.
\end{theorem}
\begin{proof}
    Given a set $S$, we let $f(S)$ denote the multiset given by applying $f$ to each element of $S$, and we let $\sum f(S)$ denote the sum over the elements of that multiset (with multiplicity). We compute
    \[
    \begin{split}
        \text{RHS} &= \sum\limits_{r \in f^{-1}(V_2(S))} \frac{f(|\phi(w) \cdot r|)}{m_{|w \cdot f(r)|}} - \sum\limits_{r \in f^{-1}(V_0(S))} \frac{f(|\phi(w) \cdot r|)}{m_{|w \cdot f(r)|}} \\
            &= \sum\limits_{s \in V_2(S)} \frac{\sum f(|\phi(w) \cdot f^{-1}(s)|)}{m_{|w \cdot s|}} - \sum\limits_{s \in V_0(S)} \frac{\sum f(|\phi(w) \cdot f^{-1}(s)|)}{m_{|w \cdot s|}} \\
            &= \sum\limits_{s \in V_2(S)} \frac{\sum f(|f^{-1}(w \cdot s)|)}{m_{|w \cdot s|}} - \sum\limits_{s \in V_0(S)} \frac{\sum f(|f^{-1}(w \cdot s)|)}{m_{|w \cdot s|}} \\
            &= \sum\limits_{s \in V_2(S)} \frac{\sum f(f^{-1}(|w \cdot s|))}{m_{|w \cdot s|}} - \sum\limits_{s \in V_0(S)} \frac{\sum f(f^{-1}(|w \cdot s|))}{m_{|w \cdot s|}} \\
            &= \sum\limits_{s \in V_2(S)} |w \cdot s| - \sum\limits_{s \in V_0(S)} |w \cdot s| \\
            &= \zeta_S(w),
    \end{split}
    \]
    as desired. 
\end{proof}
We opt for the notation $\zeta_R$ when we want to clarify which root lattice we are computing $\zeta$ of.

\begin{corollary} \label{cor:score-coeffs}
    Let $f: R \to S$ be positive and root surjective, let $\phi: W'_S \to W'_R$ be compatible with $f$, and suppose $m_s = m$ for all roots $s \in S$. Define $U_k := V_k(R) \setminus f^{-1}(V_k(S))$ for all $k$. Then for all $w \in W'_S$ and all simple roots $\beta \in S$ we have
    \[
        \zeta_S(w)_\beta = \frac{1}{m} \sum_{\alpha \in \Delta_R} f(\alpha)_\beta \left[\zeta_R(\phi(w))_\alpha - \left(\sum\limits_{r \in U_2} |\phi(w) \cdot r|_\alpha - \sum\limits_{r \in U_0} |\phi(w) \cdot r|_\alpha\right)\right],
    \]
    where $r_\alpha$ denotes the $\alpha$ coefficient of $r \in R$.
\end{corollary}
\begin{proof}
    By Theorem~\ref{thm:score} and Lemma~\ref{lem:k-roots} we have
    \[
    \begin{split}
        m \cdot \zeta_S(w) &= f\left[\sum\limits_{r \in f^{-1}(V_2(S))} |\phi(w) \cdot r| - \sum\limits_{r \in f^{-1}(V_0(S))} |\phi(w) \cdot r|\right] \\
            &= f\left[\zeta_R(\phi(w)) - \left(\sum\limits_{r \in U_2} |\phi(w) \cdot r| - \sum\limits_{r \in U_0} |\phi(w) \cdot r|\right)\right].
    \end{split}
    \]
    Thus by linearity of $y \mapsto y_\beta$, we only need to show
    \begin{equation}{\label{eqn: coefficient formula}}
         f(x)_\beta = \sum_{\alpha \in \Delta_R} f(\alpha)_\beta \cdot x_\alpha,
    \end{equation}
    for all $x \in R$. Since $x = \sum_{\alpha \in \Delta_R} x_\alpha \cdot \alpha$, the result then follows from linearity of $f$ and of $y \mapsto y_\beta$.
\end{proof}

\section{\texorpdfstring{$B_n$}{Bn} case}

We now analyze the $B_n$ case. We work with the presentation of the $B_n$ root system with short roots $\pm \bm{e}_i$ for all $i$ and long roots $\pm \bm{e}_i \pm \bm{e}_j$ for $i \neq j$, and define the ordered set $\Delta_{B_n} := \{ \bm{e}_1 - \bm{e}_2, ... , \bm{e}_{n-1} - \bm{e}_n, \bm{e}_n\}$. We also denote the elements of $\Delta_{B_n}$ as $\beta_1,\ldots,\beta_n$ in the order given above.

Let $R = \Z \cdot \Phi(\GL_{2n+1}, T)$ and $S = \Z \cdot \Phi(\SO_{2n+1}, T_{\SO_{2n+1}})$. Recall that the Weyl group of $B_n$ is the set of signed permutations of $[n]$. We embed this into $W_R = W_{A_{2n}} = S_{2n+1}$ by considering $W_S = W_{B_n}$ to be the set of permutations $\pi$ of
\begin{equation}{\label{eqn: B-type permutations}}
        \{1,2,\ldots,n-1,n,0,-n,-(n-1),\ldots,-2,-1\}
\end{equation}
for which $\pi(0) = 0$ and $\pi(i) = j$ implies $\pi(-i) = -j$. Call this embedding $\phi: W_{B_n} \to W_{A_{2n}}$. We further have a linear map $f$, mapping the root lattice of $A_{2n}$ onto the root lattice of $B_n$, given via
\[
    f(\bm{e}_i) := \begin{cases}
        \bm{e}_i, & i \leq n \\
        \bm{0}, & i = n+1 \\
        -\bm{e}_{2n+2-i}, & i \geq n+2
    \end{cases}.
\]
Note that for $1 \leq i < j \leq n$ we have, when restricted to roots,
\[
    f^{-1}(\bm{e}_i - \bm{e}_j) = \{\bm{e}_i - \bm{e}_j, \bm{e}_{2n+2-j} - \bm{e}_{2n+2-i}\}
\]
and
\[
    f^{-1}(\bm{e}_i + \bm{e}_j) = \{\bm{e}_i - \bm{e}_{2n+2-j}, \bm{e}_j - \bm{e}_{2n+2-i}\},
\]
and for $1 \leq i \leq n$ we have
\[
    f^{-1}(\bm{e}_i) = \{\bm{e}_i - \bm{e}_{n+1}, \bm{e}_{n+1} - \bm{e}_{2n+2-i}\}.
\]
(Note that $f$ does not map roots to roots; for example we have $f(\bm{e}_i-\bm{e}_{2n+2-i}) = 2\bm{e}_i$ for all $1 \leq i \leq n$.) Thus $f$ is positive and root surjective, and the embedding $\phi$ is compatible with $f$ as defined above. Note further that $m_s = 2$ for all positive roots $s$ of $B_n$.

Now fix a weight function $\rho$ on $\Delta_{B_n}$ so that $f^{-1}(\rho)$ is the associated weight function for $\Delta_{A_{2n}}$. Note that for any simple root $\beta \in \Delta_{B_n}$ there exist $\alpha_1,\alpha_2 \in \Delta_{A_{2n}}$ such that for any $\alpha \in \Delta_{A_{2n}}$, we have $f(\alpha)_\beta = 1$ if and only if $\alpha \in \{\alpha_1,\alpha_2\}$ and $f(\alpha)_\beta = 0$ otherwise, where $\alpha_1$ and $\alpha_2$ are the $i$th and $(2n+1-i)$th simple roots of $A_{2n}$ for some $i$. We let $\alpha_1$ be the left-most simple root in our ordering. Thus for any $w \in W_{B_n}$ and any simple root $\beta \in \Delta_{B_n}$, Corollary~\ref{cor:score-coeffs} implies
\[
     \zeta_{B_n}(w)_\beta = \frac{1}{2} \sum_{\alpha \in \{\alpha_1,\alpha_2\}} \left[\zeta_{A_{2n}}(\phi(w))_\alpha - \left(\sum\limits_{r \in U_2} |\phi(w) \cdot r|_\alpha - \sum\limits_{r \in U_0} |\phi(w) \cdot r|_\alpha\right)\right].
\]
Recall that $U_k := V_k(A_{2n}) \setminus f^{-1}(V_k(B_n))$ for all $k$, and note that $U_2$ and $U_0$ are precisely the roots in $V_2(A_{2n})$ and $V_0(A_{2n})$ which are ``symmetric'' about $n+1$. When $\rho$ is valued in $\{0,2\}$, we have $U_2$ is always empty, and thus we have
\[
    \zeta_{B_n}(w)_\beta = \frac{1}{2} \sum_{\alpha \in \{\alpha_1,\alpha_2\}} \left[\zeta_{A_{2n}}(\phi(w))_\alpha + \sum\limits_{r \in U_0} |\phi(w) \cdot r|_\alpha\right].
\]
Now let $H_0,\ldots,H_m$ denote the partition of $[n]$ determined by the locations of the nonzero entries of $\rho$, as in $A$-type (where $m = \#\{\beta: \rho(\beta) \neq 0\}$). However, note that we are not including $0$ (as it appears in \eqref{eqn: B-type permutations}) in the partition, so if $\rho(\beta_n) = 2$ then we have $N_m = \# H_m = 0$. Thus $N_j = \#H_j \geq 1$ for all $0 \leq j \leq m-1$ and $N_m = 0$ if and only if $\rho(\beta_n) \neq 0$. We also set $H_{-1} = \varnothing$ for ease of notation. Thus in $A_{2n}$, the sizes of the associated partition $\{H^A_j\}$ are given by 
\begin{equation*}
N^A_j = \# H^A_j = \begin{cases} 
    0 & j = -1, 2m+1 \\
    N_j & 0 \leq j < m \\
    2N_m+1 & j=m \\
    N_{2m-j} & m < j \leq 2m.
\end{cases}
\end{equation*}

We now specialize to the case where $\rho$ is valued in $\{0,2\}$. For each block $H^A_j$, we let $a_j$ denote the number of elements in block $j$ which are permuted by $\phi(w)$ to the left of $\alpha_1$ and let $b_j$ denote the number of elements in block $j$ which are permuted by $\phi(w)$ to the right of $\alpha_2$. Letting $\bm{u}_j = [a_j, N^A_j-a_j]^\top$ and $\bm{v}_j = [b_j, N^A_j-b_j]^\top$, we thus have
\[
    \zeta_{A_{2n}}(\phi(w))_{\alpha_1} = -\frac{1}{2} \sum_{j=0}^{2m+1} (\bm{u}_j-\bm{u}_{j-1})^\top
        \begin{bmatrix}
            0 & 1 \\
            1 & 0 
        \end{bmatrix}
        (\bm{u}_j-\bm{u}_{j-1})
\]
and
\[
    \zeta_{A_{2n}}(\phi(w))_{\alpha_2} = -\frac{1}{2} \sum_{j=0}^{2m+1} (\bm{v}_j-\bm{v}_{j-1})^\top
        \begin{bmatrix}
            0 & 1 \\
            1 & 0 
        \end{bmatrix}
        (\bm{v}_j-\bm{v}_{j-1}).
\]
Now $U_0$ is precisely the set of ``symmetric'' roots of $A_{2n}$ contained in the middle $m$th block of size $N^A_m=2N_m+1$. Thus there are precisely $N_m$ such positive roots. Since $\phi(w)$ acts symmetrically on the elements of $[2n+1]$, we thus have
\[
    \sum_{r \in U_0} |\phi(w) \cdot r|_{\alpha_1} = 2a_m \quad \text{and} \quad \sum_{r \in U_0} |\phi(w) \cdot r|_{\alpha_2} = 2b_m,
\]
and further that $a_j = b_{2m-j}$ and $N^A_j = N^A_{2m-j}$ for all $j$. This implies $\bm{v}_j = \bm{u}_{2m-j}$, which gives the following formula 
\begin{equation}{\label{eqn: B-type final eqn}}
        \zeta_{B_n}(w)_\beta = -\frac{1}{2} \sum_{j=0}^{2m+1} (\bm{u}_j-\bm{u}_{j-1})^\top
        \begin{bmatrix}
            0 & 1 \\ 1 & 0
        \end{bmatrix}
        (\bm{u}_j-\bm{u}_{j-1}) + 2a_m. 
\end{equation}

\subsection{Proof of Theorem \ref{thm: main} in $B$-type}
In this subsection, we complete the proof of Theorem \ref{thm: main} in the $B$-type cases. We begin with two preparatory lemmas. We note that Lemma~\ref{lem:B-wt-corresp} applies to arbitrary non-negative $\rho$.

\begin{lemma} \label{lem:matrix-ineq}
    Given any integer vector $\bm{u} \in \Z^2$ such that $u_1 + u_2 \in \{-1,0,1\}$, we have that
    \[
        \bm{u}^\top \begin{bmatrix}
            0 & 1 \\ 1 & 0
        \end{bmatrix}
        \bm{u} \leq 0
    \]
    with equality if and only if $\bm{u} \in \{\bm{0},\pm\bm{e}_1,\pm\bm{e}_2\}$.
\end{lemma}
\begin{proof}
    The expression above is equal to $2u_1u_2$. If $u_1 + u_2 \in \{-1,0,1\}$ but $\bm{u} \not\in \{\bm{0},\pm\bm{e}_1,\pm\bm{e}_2\}$ then the entries of $\bm{u}$ must be non-zero and of differing signs.
\end{proof}

\begin{lemma} \label{lem:B-wt-corresp}
    Let $\bm{a}  = (a_0,\ldots,a_{2m})$ be any integer vector for which $0 \leq a_i \leq N_i^A$ for all $i$. Then $\bm{a}$ comes from a choice of $w \in W_{B_n}$ and $\beta \in \Delta_{B_n}$ if and only if $a_i > 0$ for some $i$ and $a_i + a_{2m-i} \leq N_i^A$ for all $i$.
\end{lemma}
\begin{proof}
    $(\implies)$. Given $w \in W_{B_n}$ and $\beta \in \Delta_{B_n}$, let $\pi$ be the associated signed permutation of $\{1,\ldots,n,0,-n,\ldots,-1\}$ and let $k$ be such that $\bm{e}_k-\bm{e}_{k+1}$ is the left-most simple root of $A_{2n}$ associated to $\beta$. Thus $\sum_{i=0}^{2m} a_i = k > 0$ which implies $a_i > 0$ for some $i$. Further, $a_i$ elements of $H_i^A$ are mapped by $\pi$ to $\{1,\ldots,k\}$ for all $i$, and since $\pi$ is a signed permutation this implies $a_i$ elements of $H_{2m-i}^A$ are mapped by $\pi$ to $\{-k,\ldots,-1\}$ for all $i$. Therefore $a_i + a_{2m-i}$ elements of $H_i^A$ are mapped to $\{1,\ldots,k\} \cup \{-k,\ldots,-1\}$ for all $i$. This implies $a_i + a_{2m-i} \leq N_i^A$ for all $i$, which completes this direction of the proof.
    
    $(\impliedby)$. We let $k = \sum_{i=0}^{2m} a_i$ and define $\beta \in \Delta_{B_n}$ corresponding to the simple root $\bm{e}_k - \bm{e}_{k+1}$ of $A_{2n}$. Note that this is a valid choice of $k$ since
    \[
        1 \leq k = \sum_{i=0}^{2m} a_i \leq \frac{1}{2} \sum_{i=0}^{2m} N_i^A = n + \frac{1}{2}
    \]
    by our assumptions, which implies $1 \leq k \leq n$ since $k$ is an integer.
    
    We now construct a signed permutation $\pi$ of $(1,\ldots,n,0,-n,\ldots,-1)$ as follows. Iterating over the blocks $H_0^A,H_1^A,\ldots,H_{2m}^A$ in order from left to right, map the first $a_i$ elements of block $i$ to the left-most elements which have not yet been mapped to. Since $k \leq n$, these have all been mapped to positive elements, which implies the right-most $a_{2m-i}$ elements of block $i$ have been mapped to negative elements (since we are constructing a signed permutation). Finally, iterating over the blocks $H_0^A,H_1^A,\ldots,H_{2m}^A$ in order from left to right again, map the remaining $N_i^A-a_i-a_{2m-i}$ elements of block $i$ to the left-most elements which have not yet been mapped to. Note that this construction automatically implies that $\pi(j) = k+\ell \implies \pi(-j) = -k-\ell$ for all these remaining elements by symmetry. Thus $\pi$ is a valid signed permutation.
\end{proof}

\setcounter{MaxMatrixCols}{11}

\begin{example}
    Here's an example of the implementation of the signed permutation construction of the previous proof. Suppose $m=2, n=5$ and $N_0^A = 1, N_1^A = 2, N_2^A = 5, N_3^A = 2, N_4^A = 1$ and $a_0 = 0, a_1 = 0, a_2 = 2, a_3 = 1, a_4 = 0$. Thus the first step of the construction gives
    \[
        \begin{pmatrix}
            1 & 2 & 3 & 4 & 5 & 0 & -5 & -4 & -3 & -2 & -1 \\
              &   &   & \textcolor{blue}{1} & \textcolor{blue}{2} &   &    &    & \textcolor{blue}{3} &    &    \\
        \end{pmatrix},
    \]
    with symmetric elements filled in as
    \[
        \begin{pmatrix}
            1 & 2 & 3 & 4 & 5 & 0 & -5 & -4 & -3 & -2 & -1 \\
              &   & \textcolor{blue}{-3} & 1 & 2 &   & \textcolor{blue}{-2} & \textcolor{blue}{-1} & 3 &    &    \\
        \end{pmatrix}.
    \]
    The second step of the construction then gives
    \[
        \begin{pmatrix}
            1 & 2 & 3 & 4 & 5 & 0 & -5 & -4 & -3 & -2 & -1 \\
            \textcolor{blue}{4} & \textcolor{blue}{5} & -3 & 1 & 2 & \textcolor{blue}{0} & -2 & -1 & 3 & \textcolor{blue}{-5} & \textcolor{blue}{-4} \\
        \end{pmatrix}.
    \]
    This is a signed permutation with the desired properties.
\end{example}

We now proceed with the proof of Theorem \ref{thm: main}. We first assume $\rho$ is valued in $\{0,2\}$. Recall from \cite{BalaCarter} that $\rho$ is distinguished precisely when $N_0=1$, $N_{i-1} \leq N_i \leq N_{i-1}+1$ for $1 \leq i \leq m-1$ and 
\[
    N_m = \begin{cases}
        \frac{N_{m-1}-1}{2}, & N_{m-1} \text{ odd} \\
        \frac{N_{m-1}}{2}, & N_{m-1} \text{ even},
    \end{cases}
\]
where this last condition is equivalent to $N^A_{m-1} \leq N^A_m \leq N^A_{m-1}+1$. In particular, note that except for in the regular case, this implies $\rho(\beta_n) = 0$. We now deduce that in each distinguished case, $\zeta_{B_n}(w)_{\Delta}$ is strictly positive for all $w \in W_{B_n}$. Indeed, applying Lemma~\ref{lem:matrix-ineq} to \eqref{eqn: B-type final eqn} implies $\zeta_{B_n}(w)_\beta \geq 0$ in such cases, with equality if and only if $a_m = 0$ and $\bm{u}_i - \bm{u}_{i-1} \in \{\bm{0},\bm{e}_1,\bm{e}_2\}$ for all $1 \leq i \leq m$ and $\bm{u}_i - \bm{u}_{i-1} \in \{\bm{0},-\bm{e}_1,-\bm{e}_2\}$ for all $m+1 \leq i \leq 2m$. This implies $a_i = 0$ for all $0 \leq i \leq 2m$, which does not correspond to a choice of $w$ and $\beta$ by Lemma~\ref{lem:B-wt-corresp}. 

We now prove the converse. Suppose for some $0 \leq s \leq m$ we have $N^A_s \leq N^A_{s-1} - 1$ and let $s$ be minimal with this property and observe $N_{-1} = 0$ implies $s \geq 1$. Hence, $N_i = N^A_i \geq 1$ for all $0 \leq i \leq s-1$. Set $a_i = 1$ for $0 \leq i \leq s-1$, and set $a_i = 0$ for all other values of $i$. Thus $a_i + a_{2m-i} \leq 1 \leq N_i^A$ for all $i$, and therefore $\bm{a}$ corresponds to a choice of $w$ and $\beta$ by Lemma \ref{lem:B-wt-corresp}. We now compute
\[
\begin{split}
    \zeta_{B_n}(w)_\beta &= -\frac{1}{2} (\bm{u}_0-\bm{u}_{-1})^\top
        \begin{bmatrix}
            0 & 1 \\ 1 & 0
        \end{bmatrix}
        (\bm{u}_0-\bm{u}_{-1}) \\
        &- \frac{1}{2} (\bm{u}_s-\bm{u}_{s-1})^\top
        \begin{bmatrix}
            0 & 1 \\ 1 & 0
        \end{bmatrix}
        (\bm{u}_s-\bm{u}_{s-1}) + 2a_m \\
        &= -(N_0^A-1) + (N_s^A - N_{s-1}^A + 1) \\
        &\leq 0.
\end{split}
\]

The remaining case to consider is where $N^A_i \geq N^A_{i-1}$ for all $0 \leq i \leq m$, and for some $0 \leq s \leq m$ we have $N^A_s \geq N^A_{s-1}+2$. Let $s$ be maximal with this property. Thus $N^A_j \geq 2$ for all $s \leq j \leq 2m-s$. Now set $a_i = 0$ for $i < s$ and $i > 2m-s$, and set $a_i = 1$ for $s \leq i \leq 2m-s$. Thus $a_i + a_{2m-i} = 2 \leq N_i^A$ for all $s \leq i \leq 2m-s$ and $a_i + a_{2m-i} = 0$ otherwise, and therefore $\bm{a}$ corresponds to a choice of $w$ and $\beta$ by Lemma \ref{lem:B-wt-corresp}. We now compute
\[
\begin{split}
    \zeta_{B_n}(w)_\beta &= -\frac{1}{2} (\bm{u}_s-\bm{u}_{s-1})^\top
        \begin{bmatrix}
            0 & 1 \\ 1 & 0
        \end{bmatrix}
        (\bm{u}_s-\bm{u}_{s-1}) \\
        &- \frac{1}{2} (\bm{u}_{2m-s+1}-\bm{u}_{2m-s})^\top
        \begin{bmatrix}
            0 & 1 \\ 1 & 0
        \end{bmatrix}
        (\bm{u}_{2m-s+1}-\bm{u}_{2m-s}) + 2a_m \\
        &= 2 - (N_s^A-N_{s-1}^A-1) + (N_{s-1}^A - N_s^A + 1) \\
        &\leq 2 - 2(2-1) = 0.
\end{split}
\]

We now study the case where $\rho$ is valued in $\Q_{\geq 0}$. The weight function $f^{-1}(\rho)$ on $A_{2n}$ is fixed under the nontrivial automorphism $\sigma$ of the $A_{2n}$ Dynkin diagram. As discussed in the $A$-type case, the groupings $H^A_i$ are in various connected components and we study the action of $\sigma$ on these connected components. We say that a connected component is self-dual if it is stabilized by $\sigma$. We say an edge in the graph associated to $f^{-1}(\rho)$ and the set $V_2$ is symmetric if it is fixed by $\sigma$.

We claim that there are at most $2$ self-dual connected components. Indeed, there can be at most one pair of groupings with symmetric edges between them. If such groupings exist, they are part of a self-dual connected component. On the other hand, if there is a central grouping (stabilized by $\sigma$), then the connected component of the central grouping will be self-dual.

We now construct a $w \in W_{B_n}$ that puts the graph of $f^{-1}(\rho)$ into a certain regular form. For each pair $C_1, C_2$ of distinct connected components interchanged by $\sigma$, there is a symmetric permutation that moves the vertices in $C_1$ as far as possible to the left and the vertices of $C_2$ to the right. Then repeating this process with each such pair of components, we can end up with a graph where none of the non-self dual groupings have interleaving edges and the self-dual groupings are in the middle of the graph. If there is at most $1$ self-dual component, then there exists a weight function $\rho'$ for $B_n$ such that $f^{-1}(\rho')$ produces this graph and is valued in $\{0,2,3\}$ such that adjacent groupings in the same component are separated by weight $2$, adjacent groupings in different components are separated by weight $3$, and the unique self-dual component (if it exists) is in the center of the graph. On the other hand, if there are two self-dual components, their edges must interleave, and in this case $\rho'$ is valued in $\{0,1,2,3\}$.

\begin{example}
    Let $n=3$ and consider $\rho$ which takes values $\{\frac{1}{2},1,\frac{1}{2}\}$ on the $3$ simple roots. The weighted Dynkin diagram of $B_3$ in this case is given by:
    \[
    \begin{tikzpicture}[baseline=-2pt,scale=1,
        node/.style={circle,draw,inner sep=2pt},
        edgelabel/.style={midway,above}
    ]
    
        \node[node] (a1) at (1.5,0) {\scriptsize$\frac{1}{2}$};
        \node[node] (a2) at (3,0) {$1$};
        \node[node] (a3) at (4.5,0) {\scriptsize$\frac{1}{2}$}; 

        \draw (a1) -- (a2);
    
        \draw[double distance=2pt, shorten >=10pt] (a2) -- (a3);
        \draw[double distance=2pt,-{Latex[open,length=10pt]}, shorten >=-1pt] (a2) -- (a3);
    \end{tikzpicture}
    \]
    We now consider $f^{-1}(\rho)$, which takes values $\{\frac{1}{2},1,\frac{1}{2},\frac{1}{2},1,\frac{1}{2}\}$ on the $6$ simple roots of $A_6$. The weighted Dynkin diagram, with $\rho$ values, $V_2$ black edges, and $V_0$ red edges (of which there are none in this example), is given by:
    \[
    \begin{tikzpicture}[baseline=-2pt,scale=1,
        node/.style={circle,draw,inner sep=2pt},
        edgelabel/.style={midway,above}
    ]
        \node[draw=none,fill=none] (a0) at (0,0) {};
        \node[node] (a1) at (1.5,0) {\scriptsize$\frac{1}{2}$};
        \node[node] (a2) at (3,0) {$1$};
        \node[node] (a3) at (4.5,0) {\scriptsize$\frac{1}{2}$};
        \node[node] (a4) at (6,0) {\scriptsize$\frac{1}{2}$};
        \node[node] (a5) at (7.5,0) {$1$};
        \node[node] (a6) at (9,0) {\scriptsize$\frac{1}{2}$};
        \node[draw=none,fill=none] (a7) at (10.5,0) {};

        \draw[draw=none] (a0) -- node[edgelabel] {$\{1\}$} (a1);
        \draw (a1) -- node[edgelabel] {$\{2\}$} (a2);
        \draw (a2) -- node[edgelabel] {$\{3\}$} (a3);
        \draw (a3) -- node[edgelabel] {$\{0\}$} (a4);
        \draw (a4) -- node[edgelabel] {$\{-3\}$} (a5);
        \draw (a5) -- node[edgelabel] {$\{-2\}$} (a6);
        \draw[draw=none] (a6) -- node[edgelabel] {$\{-1\}$} (a7);

        \draw[bend left=60]($(a0)!0.5!(a1) + (0,0.6)$) to ($(a3)!0.5!(a4) + (0,0.6)$);
        \draw[bend left=60]($(a1)!0.5!(a2) + (0,0.6)$) to ($(a4)!0.5!(a5) + (0,0.6)$);
        \draw[bend left=60]($(a2)!0.5!(a3) + (0,0.6)$) to ($(a5)!0.5!(a6) + (0,0.6)$);
        \draw[bend left=60]($(a3)!0.5!(a4) + (0,0.6)$) to ($(a6)!0.5!(a7) + (0,0.6)$);
    \end{tikzpicture}
    \]
    Note that we have $3$ connected components here, with exactly one of them being self-dual. Now suppose we apply the Weyl element $w' = \left(\begin{smallmatrix}1&2&3&0&-3&-2&-1\\3&1&-2&0&2&-1&-3\end{smallmatrix}\right)$, giving:
    \[
    \begin{tikzpicture}[baseline=-2pt,scale=1,
        node/.style={circle,draw,inner sep=2pt},
        edgelabel/.style={midway,above}
    ]
        \node[draw=none,fill=none] (a0) at (0,0) {};
        \node[node] (a1) at (1.5,0) {\scriptsize$\frac{1}{2}$};
        \node[node] (a2) at (3,0) {$1$};
        \node[node] (a3) at (4.5,0) {\scriptsize$\frac{1}{2}$};
        \node[node] (a4) at (6,0) {\scriptsize$\frac{1}{2}$};
        \node[node] (a5) at (7.5,0) {$1$};
        \node[node] (a6) at (9,0) {\scriptsize$\frac{1}{2}$};
        \node[draw=none,fill=none] (a7) at (10.5,0) {};

        \draw[draw=none] (a0) -- node[edgelabel] {$\{2\}$} (a1);
        \draw (a1) -- node[edgelabel] {$\{-3\}$} (a2);
        \draw (a2) -- node[edgelabel] {$\{1\}$} (a3);
        \draw (a3) -- node[edgelabel] {$\{0\}$} (a4);
        \draw (a4) -- node[edgelabel] {$\{-1\}$} (a5);
        \draw (a5) -- node[edgelabel] {$\{3\}$} (a6);
        \draw[draw=none] (a6) -- node[edgelabel] {$\{-2\}$} (a7);

        \draw[bend left=60]($(a0)!0.5!(a1) + (0,0.6)$) to ($(a1)!0.5!(a2) + (0,0.6)$);
        \draw[bend left=60]($(a2)!0.5!(a3) + (0,0.6)$) to ($(a3)!0.5!(a4) + (0,0.6)$);
        \draw[bend left=60]($(a3)!0.5!(a4) + (0,0.6)$) to ($(a4)!0.5!(a5) + (0,0.6)$);
        \draw[bend left=60]($(a5)!0.5!(a6) + (0,0.6)$) to ($(a6)!0.5!(a7) + (0,0.6)$);
    \end{tikzpicture}
    \]
    If we consider $\rho'$ which takes values $\{2,3,2\}$ on the $3$ simple roots of $B_3$, then we obtain the same diagram, implying that the values of $\zeta$ for $\rho$ and for $\rho'$ are equal up to the action of $w'$. The diagram in this case looks like:
    \[
    \begin{tikzpicture}[baseline=-2pt,scale=1,
        node/.style={circle,draw,inner sep=2pt},
        edgelabel/.style={midway,above}
    ]
        \node[draw=none,fill=none] (a0) at (0,0) {};
        \node[node] (a1) at (1.5,0) {$2$};
        \node[node] (a2) at (3,0) {$3$};
        \node[node] (a3) at (4.5,0) {$2$};
        \node[node] (a4) at (6,0) {$2$};
        \node[node] (a5) at (7.5,0) {$3$};
        \node[node] (a6) at (9,0) {$2$};
        \node[draw=none,fill=none] (a7) at (10.5,0) {};

        \draw[draw=none] (a0) -- node[edgelabel] {$\{1\}$} (a1);
        \draw (a1) -- node[edgelabel] {$\{2\}$} (a2);
        \draw (a2) -- node[edgelabel] {$\{3\}$} (a3);
        \draw (a3) -- node[edgelabel] {$\{0\}$} (a4);
        \draw (a4) -- node[edgelabel] {$\{-3\}$} (a5);
        \draw (a5) -- node[edgelabel] {$\{-2\}$} (a6);
        \draw[draw=none] (a6) -- node[edgelabel] {$\{-1\}$} (a7);

        \draw[bend left=60]($(a0)!0.5!(a1) + (0,0.6)$) to ($(a1)!0.5!(a2) + (0,0.6)$);
        \draw[bend left=60]($(a2)!0.5!(a3) + (0,0.6)$) to ($(a3)!0.5!(a4) + (0,0.6)$);
        \draw[bend left=60]($(a3)!0.5!(a4) + (0,0.6)$) to ($(a4)!0.5!(a5) + (0,0.6)$);
        \draw[bend left=60]($(a5)!0.5!(a6) + (0,0.6)$) to ($(a6)!0.5!(a7) + (0,0.6)$);
    \end{tikzpicture}
    \]
    Here, $\rho'$ (and thus $\rho$ also) is degenerate because $\zeta_{\rho'}(1)_{\beta_2} = 0$.
\end{example}

\begin{example}
    Let $n=3$ and consider $\rho$ which takes values $\{1,0,1\}$ on the $3$ simple roots. The weighted Dynkin diagram of $B_3$ in this case is given by:
    \[
    \begin{tikzpicture}[baseline=-2pt,scale=1,
        node/.style={circle,draw,inner sep=2pt},
        edgelabel/.style={midway,above}
    ]
    
        \node[node] (a1) at (1.5,0) {$1$};
        \node[node] (a2) at (3,0) {$0$};
        \node[node] (a3) at (4.5,0) {$1$}; 

        \draw (a1) -- (a2);
    
        \draw[double distance=2pt, shorten >=10pt] (a2) -- (a3);
        \draw[double distance=2pt,-{Latex[open,length=10pt]}, shorten >=-1pt] (a2) -- (a3);
    \end{tikzpicture}
    \]
    We now consider $f^{-1}(\rho)$, which takes values $\{1,0,1,1,0,1\}$ on the $6$ simple roots of $A_6$. The weighted Dynkin diagram, with $\rho$ values, $V_2$ black edges, and $V_0$ red edges, is given by:
    \[
    \begin{tikzpicture}[baseline=-2pt,scale=1,
        node/.style={circle,draw,inner sep=2pt},
        edgelabel/.style={midway,above}
    ]
        \node[draw=none,fill=none] (a0) at (0,0) {};
        \node[node] (a1) at (1.5,0) {$1$};
        \node[node] (a2) at (3,0) {$0$};
        \node[node] (a3) at (4.5,0) {$1$};
        \node[node] (a4) at (6,0) {$1$};
        \node[node] (a5) at (7.5,0) {$0$};
        \node[node] (a6) at (9,0) {$1$};
        \node[draw=none,fill=none] (a7) at (10.5,0) {};

        \draw[draw=none] (a0) -- node[edgelabel] {$\{1\}$} (a1);
        \draw (a1) -- node[edgelabel] {$\{2$} (a2);
        \draw (a2) -- node[edgelabel] {$3\}$} (a3);
        \draw (a3) -- node[edgelabel] {$\{4\}$} (a4);
        \draw (a4) -- node[edgelabel] {$\{5$} (a5);
        \draw (a5) -- node[edgelabel] {$6\}$} (a6);
        \draw[draw=none] (a6) -- node[edgelabel] {$\{7\}$} (a7);

        \draw[bend left=60]($(a0)!0.5!(a1) + (0,0.6)$) to ($(a3)!0.5!(a4) + (0,0.6)$);
        \draw[bend left=60, dashed]($(a2)!0.5!(a3) + (0,0.6)$) to ($(a4)!0.5!(a5) + (0,0.6)$);
        \draw[bend left=60]($(a3)!0.5!(a4) + (0,0.6)$) to ($(a6)!0.5!(a7) + (0,0.6)$);
        \draw[bend left=60,color=red,thick]($(a1)!0.5!(a2) + (0,0.6)$) to ($(a2)!0.5!(a3) + (0,0.6)$);
        \draw[bend left=60,color=red,thick]($(a4)!0.5!(a5) + (0,0.6)$) to ($(a5)!0.5!(a6) + (0,0.6)$);
    \end{tikzpicture}
    \]
    Note here that the symmetric black edge is dashed due to the fact that symmetric edges (when considering \textbf{all} edges from the complete bipartite graphs and complete graphs) do not correspond to any root of the original $B_3$ being considered. In this case we have two connected components, both of which are self-dual. We note that there is no Weyl element which changes the above diagram so that it corresponds to a different $\rho'$ for $B_3$.
\end{example}

As in the $A$-type case, if $\rho'$ contains any simple root $\beta$ of weight $3$, then $\zeta_{\rho'}(1)_{\beta}=0$. Thus, it remains to analyze the case where $\rho'$ is valued in $\{0,1,2\}$ and $\rho'$ takes the value $1$. We may further assume that $\rho'$ consists of groupings $H_0$ through $H_j$ separated by $2$s, and the remaining groupings separated by $1$s. As in \eqref{eqn: B-type final eqn}, the formula for $\zeta_{\rho'}$ will involve a $\zeta_{A_{2n}}$ term and a correction term removing the contribution from symmetric roots. The correction in this case is $2a_m-(a_{m-1}+a_{m+1})$ and the full formula then becomes
\begin{align*}
    \zeta_{B_n}(w)_{\beta}  =& [\sum\limits_{i=0}^{j-1} a_i(N^A_{i+1} - a_{i+1}) + a_{i+1}(N^A_i-a_i)] +[\sum\limits_{i=j}^{2m-j-2} a_i(N^A_{i+2}-a_{i+2}) + a_{i+2}(N^A_i-a_i)]\\
    &+ [\sum\limits_{i=2m-j}^{2m-1} a_i(N^A_{i+1}-a_{i+1}) + a_{i+1}(N^A_i-a_i)] -2[\sum\limits_{i=0}^{2m} a_i(N^A_i-a_i)]\\
    &+ 2a_m - (a_{m-1} + a_{m+1}),
\end{align*}
where $j$ is the grouping before the first simple root of value $1$. Then $f^{-1}(\rho')$ will be a graph with two connected components $X_1$ and $X_2$. We let $X_1$ be the set of groupings which contains the central grouping associated to $a_m$.

Now we let $a_i = 1$ for all $i \leq m-1$ which correspond to groupings in $X_2$ and $a_i=0$ otherwise. Recalling that $N^A_{m-1} = N^A_{m+1}$, we compute
\[
    \zeta_{B_n}(w)_\beta = -(N^A_{\min\{X_2\}}-1) - (N^A_{m-1}-1) + N^A_{m+1} - 1 \leq 0.
\]
This proves that such a $\rho'$ can never be non-degenerate.

\section{\texorpdfstring{$C_n$}{Cn} case}

We now repeat the above analysis in the $C_n$ case. We work with the presentation of the $C_n$ root system with short roots $\pm \bm{e}_i \pm \bm{e}_j$ for $i \neq j$ and long roots $\pm 2\bm{e}_i$, for $1 \leq i, j \leq n$ and define the ordered set $\Delta_{C_n} := \{ \bm{e}_1 - \bm{e}_2, ... , \bm{e}_{n-1} - \bm{e}_n, 2\bm{e}_n\}$. We also denote the elements of $\Delta_{C_n}$ as $\beta_1,\ldots,\beta_n$ in the order given above.

We embed $\phi: W_{C_n} \hookrightarrow W_{A_{2n-1}}$ as signed permutations of $\{ 1, ..., n, -n, ..., -1\}$. Let $R=\Z \cdot \Phi(\GL_{2n}, T)$ and $S = \Z \cdot \Phi(\Sp_{2n}, T_{\Sp_{2n}})$. The map $f: R \to S$ is defined by
\begin{equation*}
    f(\bm{e}_i) =\begin{cases}
    \bm{e}_i & i \leq n \\
    -\bm{e}_{2n+1 -i} & i > n.
\end{cases}
\end{equation*}
Hence $f^{-1}(\bm{e}_i - \bm{e}_j) = \{ \bm{e}_i - \bm{e}_j , \bm{e}_{2n+1-j} - \bm{e}_{2n+1-i}\}$ and $f^{-1}(\bm{e}_i + \bm{e}_j) = \{ \bm{e}_i-\bm{e}_{2n+1-j}, \bm{e}_j - \bm{e}_{2n+1 - i}\}$ and $f^{-1}(2\bm{e}_i) = \{ \bm{e}_i - \bm{e}_{2n+1-i}\}$. Thus $f$ is positive and root surjective and compatible with $\phi$, but the multiplicities are not constant as opposed to the $B_n$-type case. However, $f$ does map roots to roots and hence $U_k= \varnothing$ for all $k$.

We let $V_{2,s}(S)$ denote the short roots in $V_2(S)$ and $V_{2, l}(S)$ denote the long roots. Similarly, we denote by $V_{2, s}(R)$ the elements of $V_2(R)$ that map to short roots under $f$ and analogously for the other notation. Then Theorem \ref{thm:score} gives (using that roots map to roots)
\begin{align*}
      \zeta_{C_n}(w) &= f\left[\sum\limits_{r \in V_{2,l}(R)} |\phi(w) \cdot r| + \sum\limits_{r \in V_{2,s}(R)} \frac{|\phi(w) \cdot r|}{2} \right]\\
      &- f \left[\sum\limits_{r \in V_{0,l}(R)} |\phi(w) \cdot r| + \sum\limits_{r \in V_{0,s}(R)} \frac{| \phi(w) \cdot r |}{2} \right].
\end{align*}
For any simple roots $\alpha \in \Delta_{A_{2n-1}}$ and $\beta \in \Delta_{C_n}$, we have $f(\alpha)_{\beta} = 1$ if and only if $\alpha \in f^{-1}(\beta)$ and $f(\alpha)_{\beta} = 0$ otherwise. Thus by linearity of $f$ and $(\cdot)_{\beta}$, we get
 
\begin{equation}{\label{eqn: Cn prelim formula}}
     \zeta(w)_{\beta} = \frac{1}{2} \sum\limits_{\alpha \in f^{-1}(\beta)} \left[  \zeta_{A_{2n-1}}(\phi(w))_{\alpha} + \sum\limits_{r \in V_{2,l}(R)} |\phi(w) \cdot r|_{\alpha}  -  \sum\limits_{r \in V_{0,l}(R)} |\phi(w) \cdot r|_{\alpha} \right]
\end{equation}

As before, the weight function $\rho$ on $\Delta_{C_n}$ lifts to a function $\rho_{A_{2n-1}}$ on $\Delta_{A_{2n-1}}$. For $C_n$ we define $H_0, ..., H_m$ to be a partition of $[n]$ determined by $\rho^{-1}(a)$ for non-zero $a$ as for $A$-type and $B$-type (where $m = \#\{\beta: \rho(\beta) \neq 0\}$). As in the case of $B$-type, $N_m = \#H_m = 0$ if and only if $\rho(\beta_n) \neq 0$. Then the shape of the blocks $\{H^A_i\}$ splits naturally into two cases based on whether $\rho(\beta_n)$ is non-zero or $0$. 

In the nonzero case, we have
\begin{equation*}
    N^A_j = \begin{cases} N_j & j < m \\
    N_{2m-1-j} & m \leq j \leq 2m-1.        
    \end{cases}
\end{equation*}
On the other hand, when $\rho(\beta_n) = 0$, we have 
\begin{equation*}
    N^A_j = \begin{cases} N_j & j < m \\
    2N_m & j=m \\
    N_{2m-j} & m < j \leq 2m.        
    \end{cases}
\end{equation*}

We now study each case separately, under the additional assumption that $\rho$ is valued in $\{0,2\}$.

\paragraph{The case of $\rho(\beta_n)=2$.}

In this case, $V_{0,l}(R)=\varnothing$ and $V_{2, l}(R)$ consists of  ``symmetric'' roots and there are $N^A_m=N^A_{m-1}$ of them. As before, we fix a simple root $\beta$ and for each $0 \leq j \leq 2m-1$, we let $a_j$ denote the number of elements in $H^A_j$ that are permuted by $\phi(w)$ to the left of all of $f^{-1}(\beta)$ and let $b_j$ be the number of elements permuted to the right. For the long roots, the contribution is $a_{m-1}+a_m = b_m+b_{m-1}$. Hence \eqref{eqn: Cn prelim formula} becomes
\begin{equation*}
      \zeta(w)_{\beta} = \frac{1}{2} \sum\limits_{\alpha \in f^{-1}(\beta)} \left[  \zeta_{A_{2n-1}}(\phi(w))_{\alpha} + a_{m-1} + a_m \right].
\end{equation*}
Define $\bm{u}_j = [a_j, N^A_j-a_j]^\top$ as in the $B$-type case. When $\beta$ is short, we use $b_j=a_{2m-1-j}$ to get
\begin{equation}{\label{eqn: Cn final formula 1}}
    \zeta_{C_n}(w)_{\beta} = -\frac{1}{2} \sum_{j=0}^{2m} (\bm{u}_j-\bm{u}_{j-1})^\top
        \begin{bmatrix}
            0 & 1 \\ 1 & 0
        \end{bmatrix}
        (\bm{u}_j-\bm{u}_{j-1}) + (a_{m-1} + a_m). 
\end{equation}
When $\beta$ is long, we thus have
\begin{equation}{\label{eqn: Cn final formula 2}}
    \zeta_{C_n}(w)_{\beta} = -\frac{1}{4} \sum_{j=0}^{2m} (\bm{u}_j-\bm{u}_{j-1})^\top
        \begin{bmatrix}
            0 & 1 \\ 1 & 0
        \end{bmatrix}
        (\bm{u}_j-\bm{u}_{j-1}) + \frac{a_{m-1}+a_m}{2}.
\end{equation}

\paragraph{The case of $\rho(\beta_n) = 0$.}

In this case, $V_{2, l}(R) = \emptyset$ and $V_{0, l}(R)$ contains $N^A_m$ many such roots. As before, we fix a simple root $\beta$ and for each $0 \leq j \leq 2m$, we let $a_j$ denote the number of elements in $H^A_j$ that are permuted by $\phi(w)$ to the left of all of $f^{-1}(\beta)$ and let $b_j$ be the elements permuted to the right. For the long roots, the contribution is $-2a_m = -2b_m$ and hence
\begin{equation*}
      \zeta(w)_{\beta} = \frac{1}{2} \sum\limits_{\alpha \in f^{-1}(\beta)} \left[  \zeta_{A_{2n-1}}(\phi(w))_{\alpha} - 2a_m \right].
\end{equation*}
When $\beta$ is short,  we use $b_j=a_{2m-j}$ to get
\begin{equation}{\label{eqn: Cn final formula 3}}
    \zeta_{C_n}(w)_{\beta} = -\frac{1}{2} \sum_{j=0}^{2m+1} (\bm{u}_j-\bm{u}_{j-1})^\top
        \begin{bmatrix}
            0 & 1 \\ 1 & 0
        \end{bmatrix}
        (\bm{u}_j-\bm{u}_{j-1}) -2a_m.
\end{equation}
When $\beta$ is long, we thus have
\begin{equation}{\label{eqn: Cn final formula 4}}
    \zeta_{C_n}(w)_{\beta} = -\frac{1}{4} \sum_{j=0}^{2m+1} (\bm{u}_j-\bm{u}_{j-1})^\top
        \begin{bmatrix}
            0 & 1 \\ 1 & 0
        \end{bmatrix}
        (\bm{u}_j-\bm{u}_{j-1}) -a_m.
\end{equation}

\subsection{Proof of Theorem \ref{thm: main} in $C$-type}

In this subsection, we complete the proof of Theorem \ref{thm: main} in the $C$-type cases. We utilize Lemma~\ref{lem:matrix-ineq} and the following.

\begin{lemma} \label{lem:C-wt-corresp}
    Fix a weight function $\rho$, and let $m^A = 2m-1$ or $m^A = 2m$ depending on whether $\rho(\beta_n)$ is positive or $0$, respectively. Let $\bm{a} = (a_0,\ldots,a_{m^A})$ be any integer vector for which $0 \leq a_i \leq N_i^A$ for all $i$. Then $\bm{a}$ comes from a choice of $w \in W_{C_n}$ and $\beta \in \Delta_{C_n}$ if and only if $a_i > 0$ for some $i$ and $a_i + a_{m^A-i} \leq N_i^A$ for all $i$.
\end{lemma}
\begin{proof}
    The proof of this fact is essentially the same as that of Lemma~\ref{lem:B-wt-corresp}. The difference is that here we consider signed permutations of $\{1,\ldots,n,-n,\ldots,-1\}$ (instead of $\{1,\ldots,n,0,-n,\ldots,-1\}$). Thus each $\beta \in \Delta_{C_n}$ corresponds to two simple roots of $A_{2n-1}$, except for the long simple root in $\Delta_{C_n}$ which corresponds to $\bm{e}_n-\bm{e}_{n+1} \in \Delta_{A_{2n-1}}$. And further, we have $\frac{1}{2} \sum_{i=0}^{m^A} N_i^A = n$. With these changes, the proof of Lemma~\ref{lem:B-wt-corresp} proves the result.
\end{proof}

We now proceed with the proof of Theorem \ref{thm: main}. Recall from \cite{BalaCarter} that $\rho$ is distinguished precisely when $N_0=1$, $N_{i-1} \leq N_i \leq N_{i-1}+1$ for $1 \leq i \leq m-1$, and $\rho(\beta_n) = 2$. We now deduce that in each distinguished case, $\zeta_{C_n}(w)_{\Delta}$ is strictly positive for all $w \in W_{C_n}$. Indeed, this follows from applying  Lemma~\ref{lem:matrix-ineq} to \eqref{eqn: Cn final formula 1} and \eqref{eqn: Cn final formula 2} since we have $\zeta_{C_n}(w)_\beta \geq 0$ in such cases, with equality if and only if $a_{m-1} = a_m = 0$ and $\bm{u}_i - \bm{u}_{i-1} \in \{\bm{0},\bm{e}_1,\bm{e}_2\}$ for all $0 \leq i \leq m-1$ and $\bm{u}_m-\bm{u}_{m-1} = \bm{0}$ and $\bm{u}_i - \bm{u}_{i-1} \in \{\bm{0},-\bm{e}_1,-\bm{e}_2\}$ for all $m+1 \leq i \leq 2m$. This implies $a_i = 0$ for all $0 \leq i \leq 2m-1$, which does not correspond to a choice of $w$ and $\beta$ by Lemma~\ref{lem:C-wt-corresp}.

We now prove the converse. First we assume $\rho(\beta_n) = 2$. Suppose for some $0 \leq s \leq m-1$ we have $N^A_s \leq N^A_{s-1} - 1$ and let $s$ be minimal with this property and observe $N_{-1} = 0$ implies $s \geq 1$. Hence, $N_i = N^A_i \geq 1$ for all $0 \leq i \leq s-1$. Set $a_i = 1$ for $0 \leq i \leq s-1$, and set $a_i = 0$ for all other values of $i$. Thus $a_i + a_{2m-1-i} \leq 1 \leq N_i^A$ for all $i$, and therefore $\bm{a}$ corresponds to a choice of $w$ and $\beta$ by Lemma \ref{lem:C-wt-corresp}. Up to a factor of $2$ when $\beta$ is a long root (compare \eqref{eqn: Cn final formula 1} and \eqref{eqn: Cn final formula 2}), we now compute
\[
\begin{split}
    \zeta_{C_n}(w)_\beta &= -\frac{1}{2} (\bm{u}_0-\bm{u}_{-1})^\top
        \begin{bmatrix}
            0 & 1 \\ 1 & 0
        \end{bmatrix}
        (\bm{u}_0-\bm{u}_{-1}) \\
        &- \frac{1}{2} (\bm{u}_s-\bm{u}_{s-1})^\top
        \begin{bmatrix}
            0 & 1 \\ 1 & 0
        \end{bmatrix}
        (\bm{u}_s-\bm{u}_{s-1}) + (a_{m-1}+a_m) \\
        &= -(N_0^A-1) + (N_s^A - N_{s-1}^A + 1) + 0 \\
        &\leq 0.
\end{split}
\]

The remaining case to consider when $\rho(\beta_n) = 2$ is where $N^A_i \geq N^A_{i-1}$ for all $0 \leq i \leq m-1$, and for some $0 \leq s \leq m-1$ we have $N^A_s \geq N^A_{s-1}+2$. Let $s$ be maximal with this property. Thus $N^A_j \geq 2$ for all $s \leq j \leq 2m-1-s$. Now set $a_i = 0$ for $i < s$ and $i > 2m-1-s$, and set $a_i = 1$ for $s \leq i \leq 2m-1-s$. Thus $a_i + a_{2m-1-i} = 2 \leq N_i^A$ for all $s \leq i \leq 2m-1-s$ and $a_i + a_{2m-1-i} = 0$ otherwise, and therefore $\bm{a}$ corresponds to a choice of $w$ and $\beta$ by Lemma \ref{lem:C-wt-corresp}. Up to a factor of $2$ when $\beta$ is a long root (compare \eqref{eqn: Cn final formula 1} and \eqref{eqn: Cn final formula 2}), we compute
\[
\begin{split}
    \zeta_{C_n}(w)_\beta &= -\frac{1}{2} (\bm{u}_s-\bm{u}_{s-1})^\top
        \begin{bmatrix}
            0 & 1 \\ 1 & 0
        \end{bmatrix}
        (\bm{u}_s-\bm{u}_{s-1}) \\
        &- \frac{1}{2} (\bm{u}_{2m-s}-\bm{u}_{2m-1-s})^\top
        \begin{bmatrix}
            0 & 1 \\ 1 & 0
        \end{bmatrix}
        (\bm{u}_{2m-s}-\bm{u}_{2m-1-s}) + (a_{m-1}+a_m) \\
        &= -(N_s^A-N_{s-1}^A-1) + (N_{s-1}^A - N_s^A + 1) + 2 \\
        &\leq -2(2-1) + 2 = 0.
\end{split}
\]

Now we assume $\rho(\beta_n) = 0$. Let $s \leq m$ be minimal such that $N_j^A \geq 2$ for all $s \leq j \leq m$. (Since $N_m^A \geq 2$ in this case, this is always possible.) Set $a_i = 1$ for all $s \leq i \leq 2m-s$, and set $a_i = 0$ for all other values of $i$. Thus $a_i + a_{2m-i} = 2 \leq N_i^A$ for all $s \leq i \leq 2m-s$ and $a_i + a_{2m-i} = 0$ otherwise, and therefore $\bm{a}$ corresponds to a choice of $w$ and $\beta$ by Lemma~\ref{lem:C-wt-corresp}. Up to a factor of $2$ when $\beta$ is a long root (compare \eqref{eqn: Cn final formula 3} and \eqref{eqn: Cn final formula 4}), we compute
\[
\begin{split}
    \zeta_{C_n}(w)_\beta &= -\frac{1}{2} (\bm{u}_s-\bm{u}_{s-1})^\top
        \begin{bmatrix}
            0 & 1 \\ 1 & 0
        \end{bmatrix}
        (\bm{u}_s-\bm{u}_{s-1}) \\
        &- \frac{1}{2} (\bm{u}_{2m-s+1}-\bm{u}_{2m-s})^\top
        \begin{bmatrix}
            0 & 1 \\ 1 & 0
        \end{bmatrix}
        (\bm{u}_{2m-s+1}-\bm{u}_{2m-s}) - 2a_m \\
        &= -(N_s^A-N_{s-1}^A-1) + (N_{s-1}^A - N_s^A + 1) - 2 \\
        &= -2(N^A_s - N^A_{s-1}) \leq -2. 
\end{split}
\]

We now analyze the case where $\rho$ is valued in $\Q_{\geq 0}$. As in the $B$-type case, the diagram attached to $f^{-1}(\rho)$ has at most $2$ self-dual components and $V_2$ can be permuted to some $w(V_2)$ coming from $\rho'$ valued in $\{0,2,3\}$ unless there are exactly $2$ self-dual components, in which case $\rho'$ will be valued in $\{0,1,2,3\}$. Again, if $\rho'$ takes value $3$, it is degenerate. Hence we can assume that $\rho'$ is valued in $\{0,1,2\}$. In the $C$-case, two self-dual components further implies that $\rho'(\beta_n)=0$. As before, we can compute a formula for $\zeta_{\rho'}(w')$ for $w' \in W_{C_n}$. The correction term (for short roots) is then $a_{m-1} + a_{m+1} - 2a_m$. We again consider the two components $X_1, X_2$ of the graph associated to $f^{-1}(\rho')$ where $X_1$ contains the central grouping.

Index the groupings of $X_1$ via
\[
    x_{-T} < \cdots < x_{-1} < x_0 = m < x_1 < \cdots < x_T,
\]
and choose $t \geq 0$ to be maximal such that $N_{x_j}^A \geq 2$ for all $|j| \leq t$. (Note that $t \geq 0$ is legitimate since $N_m^A \geq 2$.) Set $a_{x_j} = 1$ for $|j| \leq t$ and set all other $a_i = 0$. This is admissible by Lemma~\ref{lem:C-wt-corresp}, and corresponds to a short root since $X_2$ is non-empty. Using the previous formulas, the type-$A$ contribution for this choice of $a_i$ is then
\[
    2\left(N_{x_{t+1}}^A - N_{x_t}^A\right) + 2,
\]
and the correction term is $-2$ since $a_{m-1} = a_{m+1} = 0$. Thus
\[
    \zeta_{\rho'}(w)_\beta = 2\left(N_{x_{t+1}}^A - N_{x_t}^A\right) \leq -2,
\]
where we use the convention that $N_{x_{T+1}}^A = 0$. This follows since $N_{x_{t+1}}^A = N_{x_{-t-1}}^A \leq 1$ and $N_{x_t}^A = N_{x_{-t}}^A \geq 2$ by maximality, and thus $\rho'$ is degenerate.

\section{\texorpdfstring{$D_n$}{Dn} case}

We now repeat the above analysis in the $D_n$ case. We work with the presentation of the $D_n$ root system with roots $\pm \bm{e}_i \pm \bm{e}_j$ for $i \neq j$, for $1 \leq i, j \leq n$ and define the ordered set $\Delta_{D_n} := \{ \bm{e}_1 - \bm{e}_2, ... , \bm{e}_{n-1} - \bm{e}_n, \bm{e}_{n-1} + \bm{e}_n\}$. We also denote the elements of $\Delta_{D_n}$ as $\beta_1,\ldots,\beta_n$ in the order given above.

Recall that the Weyl group of $D_n$ is the set of signed permutations of $[n]$ with evenly many sign flips. Let $W'_{D_n}$ be the group of signed permutations of $[n]$ so that naturally $W_{D_n} \subset W'_{D_n}$. We embed $W'_{D_n}$ into $W_{A_{2n-1}} = S_{2n}$ by considering $W'_{D_n}$ to be the set of permutations $\pi$ of
\[
    \{1,2,\ldots,n-1,n,-n,-(n-1),\ldots,-2,-1\}
\]
for which $\pi(i) = j$ implies $\pi(-i) = -j$. Call this embedding $\phi: W'_{D_n} \to W_{A_{2n-1}}$. We further have a linear map $f$, mapping the root lattice of $A_{2n-1}$ onto the root lattice of $D_n$, given via
\[
    f(\bm{e}_i) := \begin{cases}
        \bm{e}_i, & i \leq n \\
        -\bm{e}_{2n+1-i}, & i \geq n+1
    \end{cases}.
\]
Further we have, when restricted to roots,
\begin{align*}
    f^{-1}(\bm{e}_i - \bm{e}_j) &= \{\bm{e}_i - \bm{e}_j, \bm{e}_{2n+1-j} - \bm{e}_{2n+1-i}\}, &\text{ for } 1 \leq i < j \leq n\\
    f^{-1}(\bm{e}_i + \bm{e}_j) &= \{\bm{e}_i - \bm{e}_{2n+1-j}, \bm{e}_j - \bm{e}_{2n+1-i}\}, &\text{ for } 1 \leq i < j \leq n.
\end{align*}
Thus $f$ is positive and root surjective, and the embedding $\phi$ is compatible with $f$ as defined above. Note further that $m_s = 2$ for all positive roots $s \in D_n$.

To simplify the computations, note that $\zeta_{D_n}(w)$ still makes sense for all $w \in W'_{D_n}$ since $W'_{D_n}$ preserves the roots of $D_n$. Given $w \in W_{D_n}$, we will denote by $\tilde{w}$ the signed permutation given by applying $w$ followed by flipping the sign of $n$ (the outer automorphism of the Dynkin diagram of $D_n$). With this, we have the following, which implies we can reduce to computing the coefficients of the simple roots $\bm{e}_1-\bm{e}_2, \ldots, \bm{e}_{n-2}-\bm{e}_{n-1},\bm{e}_{n-1}+\bm{e}_n$ (i.e. all but one tail node) for all signed permutations.

\begin{lemma}
    Fix $w \in W_{D_n}$ and simple root $\beta \in D_n$. Then
    \[
        \zeta_{D_n}(w)_\beta = \begin{cases}
            \zeta_{D_n}(\tilde{w})_\beta, & \beta \in \{\bm{e}_1-\bm{e}_2, \ldots, \bm{e}_{n-2} - \bm{e}_{n-1}\} \\
            \zeta_{D_n}(\tilde{w})_{\bm{e}_{n-1}+\bm{e}_n}, & \beta = \bm{e}_{n-1} - \bm{e}_n \\
            \zeta_{D_n}(\tilde{w})_{\bm{e}_{n-1}-\bm{e}_n}, & \beta = \bm{e}_{n-1} + \bm{e}_n.
        \end{cases}
    \]
\end{lemma}
\begin{proof}
    Fix any root $r \in D_n$ and any $w \in W_{D_n}$. Note that $w(r)$ is positive if and only if $\tilde{w}(r)$ is positive. If $|w(r)| = \bm{e}_i \pm \bm{e}_j$ for $i < j < n$ then $|\tilde{w}(r)| = |w(r)|$. And if $|w(r)| = \bm{e}_i \pm \bm{e}_n$ for $i < n$ then $|\tilde{w}(r)| = \bm{e}_i \mp \bm{e}_n$, which implies $|\tilde{w}(r)| = |w(r)| \pm [(\bm{e}_{n-1}-\bm{e}_n)-(\bm{e}_{n-1}+\bm{e}_n)] = |w(r)| \pm (\beta_{n-1}-\beta_n)$. Thus for any simple root $\beta \in \Delta_{D_n}$ we have
    \[
        |\tilde{w}(r)|_\beta = \begin{cases}
            |w(r)|_\beta, & \beta \in \{\bm{e}_1-\bm{e}_2,\ldots,\bm{e}_{n-2}-\bm{e}_{n-1}\} \\
            |w(r)|_{\bm{e}_{n-1}+\bm{e}_n}, & \beta = \bm{e}_{n-1}-\bm{e}_n \\
            |w(r)|_{\bm{e}_{n-1}-\bm{e}_n}, & \beta = \bm{e}_{n-1}+\bm{e}_n.
        \end{cases}
    \]
    The result follows. One way to see the last two rows is to notice that the difference between $|w(r)|_{\beta_n}$ and $|w(r)|_{\beta_{n-1}}$ is given by the $\bm{e}_n$-coefficient of $|w(r)|$ in the $\{\bm{e}_i\}$-basis. Hence, for instance we have for $|w(r)| = \bm{e}_i \pm \bm{e}_n$ that 
    \begin{align*}
        |\tilde{w}(r)|_{\beta_n} - |w(r)|_{\beta_{n-1}} &= (|\tilde{w}(r)|_{\beta_n} - |w(r)|_{\beta_n}) + (|w(r)|_{\beta_n} - |w(r)|_{\beta_{n-1}})\\
        &= (\mp 1) + (\pm 1)=0.
    \end{align*}
\end{proof}

\begin{corollary}
    Fix a tail root $\beta' \in D_n$. We have that $\zeta_{D_n}(w)_\beta > 0$ for all $w \in W_{D_n}$ and all simple roots $\beta \in D_n$ if and only if $\zeta_{D_n}(w)_\beta> 0$ for all $w \in W'_{D_n}$ and all simple roots $\beta \in D_n$ such that $\beta \neq \beta'$.
\end{corollary}

Now let $\rho$ be a weighting of the simple roots of $D_n$. To ensure that $f^{-1}(\rho)[\alpha] \geq 0$ for all simple roots $\alpha \in A_{2n-1}$, we assume that $\rho(\beta_n) \geq \rho(\beta_{n-1})$ by applying the non-trivial outer automorphism of the Dynkin diagram $\theta$ to $\rho$ if necessary. This is without loss of generality because applying $\theta \in W_{D_n}'$ to the roots of $D_n$ and the weight function $\rho$ shows that $\rho$ is distinguished if and only if $\theta(\rho)$ is distinguished, and that the desired positivity condition holds for $\rho$ if and only if it holds for $\theta(\rho)$. As above, define $U_k := V_k(A_{2n-1}) \setminus f^{-1}(V_k(D_n))$ for all $k$. Thus for any $w \in W'_{D_n}$ and any simple root $\beta \in D_n$, by Corollary~\ref{cor:score-coeffs} we have
\begin{equation} \label{eqn: Dn prelim formula}
    \zeta_{D_n}(w)_\beta = \sum_\alpha \frac{f(\alpha)_\beta}{2} \left[\zeta_{A_{2n-1}}(\phi(w))_\alpha - \left(\sum\limits_{r \in U_2} |\phi(w) \cdot r|_\alpha - \sum\limits_{r \in U_0} |\phi(w) \cdot r|_\alpha\right)\right].
\end{equation}
Note that $U_2$ and $U_0$ are precisely the roots in $V_2(A_{2n-1})$ and $V_0(A_{2n-1})$ which are ``symmetric''; that is, those that have the form $\bm{e}_i - \bm{e}_{2n+1-i}$ for some $i$. Note also that for $\beta \not\in \{\bm{e}_{n-1}-\bm{e}_n, \bm{e}_{n-1}+\bm{e}_n\}$, we have that $f(\alpha)_\beta = 1$ for two simple roots $\alpha_1$ and $\alpha_2$ and $f(\alpha)_\beta = 0$ otherwise. For $\beta = \bm{e}_{n-1}+\bm{e}_n$, we have that $f(\alpha)_\beta = 1$ for one simple root $\alpha_0 = \bm{e}_n-\bm{e}_{n+1}$ and $f(\alpha)_\beta = 0$ otherwise. Finally for $\beta = \bm{e}_{n-1}-\bm{e}_n$, we have that $f(\alpha)_\beta = 1$ for two simple roots $\alpha_1 = \bm{e}_{n-1}-\bm{e}_n$ and $\alpha_2 = \bm{e}_{n+1} - \bm{e}_{n+2}$, $f(\alpha)_\beta = -1$ for one simple root $\alpha_0 = \bm{e}_n-\bm{e}_{n+1}$, and $f(\alpha)_\beta = 0$ otherwise.

As before, the weight function $\rho$ on $\Delta_{D_n}$ lifts to a function $\rho_{A_{2n-1}}$ on $\Delta_{A_{2n-1}}$. For $D_n$ we define $H_0, ..., H_m$ to be a partition of $[n]$ determined by the locations of the non-zero entries of $\rho$, as for the other types (where $m = \#\{\beta: \rho(\beta) \neq 0\}$). As in the case of $B$-type and $C$-type, $N_m = \#H_m = 0$ if and only if $\rho(\beta_n) \neq 0$. For the following block descriptions and coefficient formulas, assume that $\rho$ is valued in $\{0,2\}$. Then the shape of the blocks $\{H^A_i\}$ splits naturally into three cases based on whether $(\rho(\beta_{n-1}),\rho(\beta_n))$ is $(2,2)$, $(0,0)$, or $(0,2)$ (recall we assume $\rho(\beta_n) \geq \rho(\beta_{n-1})$). In the case where $\rho(\beta_{n-1}) = \rho(\beta_n) \neq 0$, we have
\begin{equation*}
    N^A_j = \begin{cases} N_j & j < m-1 \\
    2N_{m-1} = 2 & j=m-1 \\
    N_{2m-2-j} & m \leq j \leq 2m-2.        
    \end{cases}
\end{equation*}
In the case where $\rho(\beta_{n-1}) = \rho(\beta_n) = 0$, we have
\begin{equation*}
    N^A_j = \begin{cases} N_j & j < m \\
    2N_m & j=m \\
    N_{2m-j} & m < j \leq 2m.        
    \end{cases}
\end{equation*}
In the case where $\rho(\beta_{n-1}) = 0$ and $\rho(\beta_n) \neq 0$, we have
\begin{equation*}
    N^A_j = \begin{cases} N_j & j < m \\
    N_{2m-1-j} & m \leq j \leq 2m-1.        
    \end{cases}
\end{equation*}

We now study each case separately.

\paragraph{The case $\rho(\beta_{n-1}) = \rho(\beta_n) = 2$.}

In this case, $U_2$ is empty and $U_0 = \{\pm(\bm{e}_n - \bm{e}_{n+1})\}$. Fix a simple root $\beta \in \Delta_{D_n}$ such that $\beta$ is not a tail simple root, and fix $w \in W_{D_n}'$. For each $0 \leq j \leq 2m-2$, we let $a_j$ denote the number of elements in $H_j^A$ which are permuted by $\phi(w)$ to the left of $\alpha_1$ and let $b_j$ denote the number of elements in $H_j^A$ which are permuted by $\phi(w)$ to the right of $\alpha_2$. For the two roots in $U_0$, the contribution is $2a_{m-1} = 2b_{m-1}$. Hence \eqref{eqn: Dn prelim formula} becomes
\[
    \zeta_{D_n}(w)_\beta = \frac{1}{2} \sum_{\alpha \in \{\alpha_1,\alpha_2\}} [\zeta_{A_{2n-1}}(\phi(w))_\alpha + 2a_{m-1}]
\]
Defining $\bm{u}_j = [a_j, N_j^A-a_j]^\top$ as above, $b_j = a_{2m-2-j}$ implies
\begin{equation} \label{eqn: Dn final formula 22 1}
    \zeta_{D_n}(w)_\beta = -\frac{1}{2} \sum_{j=0}^{2m-1} (\bm{u}_j-\bm{u}_{j-1})^\top
        \begin{bmatrix}
            0 & 1 \\
            1 & 0 
        \end{bmatrix}
        (\bm{u}_j-\bm{u}_{j-1})
        + 2a_{m-1}.
\end{equation}
For $\beta = \beta_n$ (recall we do not explicitly compute the $\beta_{n-1}$ coefficient), we let $a_j$ denote the number of elements in $H_j^A$ which are permuted by $\phi(w)$ to the left of $\alpha_0$. Thus \eqref{eqn: Dn prelim formula} becomes
\[
    \zeta_{D_n}(w)_\beta = \frac{1}{2} \sum_{\alpha \in \{\alpha_0\}} [\zeta_{A_{2n-1}}(\phi(w))_\alpha + 2a_{m-1}],
\]
and thus we have
\begin{equation} \label{eqn: Dn final formula 22 2}
    \zeta_{D_n}(w)_\beta = -\frac{1}{4} \sum_{j=0}^{2m-1} (\bm{u}_j-\bm{u}_{j-1})^\top
        \begin{bmatrix}
            0 & 1 \\
            1 & 0 
        \end{bmatrix}
        (\bm{u}_j-\bm{u}_{j-1})
        + a_{m-1}.
\end{equation}
\paragraph{The case $\rho(\beta_{n-1}) = \rho(\beta_n) = 0$.}

In this case, $U_2$ is empty and $U_0$ contains $N_m^A$ roots. Fix a simple root $\beta \in \Delta_{D_n}$ such that $\beta$ is not a tail simple root, and fix $w \in W_{D_n}'$. For each $0 \leq j \leq 2m$, we let $a_j$ denote the number of elements in $H_j^A$ which are permuted by $\phi(w)$ to the left of $\alpha_1$ and let $b_j$ denote the number of elements in $H_j^A$ which are permuted by $\phi(w)$ to the right of $\alpha_2$. For the roots in $U_0$, the contribution is $2a_m = 2b_m$. Hence \eqref{eqn: Dn prelim formula} becomes
\[
    \zeta_{D_n}(w)_\beta = \frac{1}{2} \sum_{\alpha \in \{\alpha_1,\alpha_2\}} [\zeta_{A_{2n-1}}(\phi(w))_\alpha + 2a_m]
\]
Defining $\bm{u}_j = [a_j, N_j^A-a_j]^\top$ as above, $b_j = a_{2m-j}$ implies
\begin{equation} \label{eqn: Dn final formula 00 1}
    \zeta_{D_n}(w)_\beta = -\frac{1}{2} \sum_{j=0}^{2m+1} (\bm{u}_j-\bm{u}_{j-1})^\top
        \begin{bmatrix}
            0 & 1 \\
            1 & 0 
        \end{bmatrix}
        (\bm{u}_j-\bm{u}_{j-1})
        + 2a_m.
\end{equation}
For $\beta = \beta_n$ (recall we do not explicitly compute the $\beta_{n-1}$ coefficient), we let $a_j$ denote the number of elements in $H_j^A$ which are permuted by $\phi(w)$ to the left of $\alpha_0$. Thus \eqref{eqn: Dn prelim formula} becomes
\[
    \zeta_{D_n}(w)_\beta = \frac{1}{2} \sum_{\alpha \in \{\alpha_0\}} [\zeta_{A_{2n-1}}(\phi(w))_\alpha + 2a_m],
\]
and thus we have
\begin{equation} \label{eqn: Dn final formula 00 2}
    \zeta_{D_n}(w)_\beta = -\frac{1}{4} \sum_{j=0}^{2m+1} (\bm{u}_j-\bm{u}_{j-1})^\top
        \begin{bmatrix}
            0 & 1 \\
            1 & 0 
        \end{bmatrix}
        (\bm{u}_j-\bm{u}_{j-1})
        + a_m.
\end{equation}
Note that this case is essentially the same as the $\rho(\beta_{n-1}) = \rho(\beta_n) = 2$ case, up to re-indexing $N_j^A$.
\paragraph{The case $\rho(\beta_{n-1}) = 0$ and $\rho(\beta_n) = 2$.}

In this case, $U_0$ is empty and $U_2$ contains $N_{m-1}^A = N_m^A$ roots. Fix a simple root $\beta \in \Delta_{D_n}$ such that $\beta$ is not a tail simple root, and fix $w \in W_{D_n}'$. For each $0 \leq j \leq 2m-1$, we let $a_j$ denote the number of elements in $H_j^A$ which are permuted by $\phi(w)$ to the left of $\alpha_1$ and let $b_j$ denote the number of elements in $H_j^A$ which are permuted by $\phi(w)$ to the right of $\alpha_2$. For the roots in $U_2$, the contribution is $-(a_{m-1}+a_m) = -(b_m+b_{m-1})$. Hence \eqref{eqn: Dn prelim formula} becomes
\[
    \zeta_{D_n}(w)_\beta = \frac{1}{2} \sum_{\alpha \in \{\alpha_1,\alpha_2\}} [\zeta_{A_{2n-1}}(\phi(w))_\alpha - (a_{m-1}+a_m)]
\]
Defining $\bm{u}_j = [a_j, N_j^A-a_j]^\top$ as above, $b_j = a_{2m-1-j}$ implies
\begin{equation} \label{eqn: Dn final formula 02 1}
    \zeta_{D_n}(w)_\beta = -\frac{1}{2} \sum_{j=0}^{2m} (\bm{u}_j-\bm{u}_{j-1})^\top
        \begin{bmatrix}
            0 & 1 \\
            1 & 0 
        \end{bmatrix}
        (\bm{u}_j-\bm{u}_{j-1})
        - (a_{m-1}+a_m).
\end{equation}
For $\beta = \beta_n$ (recall we do not explicitly compute the $\beta_{n-1}$ coefficient), we let $a_j$ denote the number of elements in $H_j^A$ which are permuted by $\phi(w)$ to the left of $\alpha_0$. Thus \eqref{eqn: Dn prelim formula} becomes
\[
    \zeta_{D_n}(w)_\beta = \frac{1}{2} \sum_{\alpha \in \{\alpha_0\}} [\zeta_{A_{2n-1}}(\phi(w))_\alpha - (a_{m-1}+a_m)],
\]
and thus we have
\begin{equation} \label{eqn: Dn final formula 02 2}
    \zeta_{D_n}(w)_\beta = -\frac{1}{4} \sum_{j=0}^{2m} (\bm{u}_j-\bm{u}_{j-1})^\top
        \begin{bmatrix}
            0 & 1 \\
            1 & 0 
        \end{bmatrix}
        (\bm{u}_j-\bm{u}_{j-1})
        - \frac{a_{m-1}+a_m}{2}.
\end{equation}
\subsection{Proof of Theorem \ref{thm: main} in $D$-type}

In this subsection, we complete the proof of Theorem \ref{thm: main} in the $D$-type cases. We utilize Lemma~\ref{lem:matrix-ineq} and the following.

\begin{lemma} \label{lem:D-wt-corresp}
    Fix a weight function $\rho$, and let $m^A$ be such that $m^A + 1$ is the number of groupings in the lifted $A$-diagram.
    Let $\bm{a} = (a_0,\ldots,a_{m^A})$ be any integer vector for which $0 \leq a_i \leq N_i^A$ for all $i$. Then $\bm{a}$ comes from a choice of $w \in W_{D_n}'$ and $\beta \in \Delta_{D_n}$ such that $\beta \neq \beta_{n-1}$ if and only if $a_i > 0$ for some $i$, $a_i + a_{m^A-i} \leq N_i^A$ for all $i$, and $\sum_{i=0}^{m^A} a_i \neq n-1$.
\end{lemma}
\begin{proof}
    The proof of this fact is essentially the same as that of Lemma~\ref{lem:B-wt-corresp}. The first difference is that here we consider signed permutations of $\{1,\ldots,n,-n,\ldots,-1\}$ (instead of $\{1,\ldots,n,0,-n,\ldots,-1\}$). Thus each $\beta \in \Delta_{D_n}$ such that $\beta \neq \beta_{n-1}$ corresponds to two simple roots of $A_{2n-1}$, except for $\beta_n \in \Delta_{D_n}$ which corresponds to $\bm{e}_n-\bm{e}_{n+1} \in \Delta_{A_{2n-1}}$. And further, we have $\frac{1}{2} \sum_{i=0}^{m^A} N_i^A = n$. 

    The second difference comes from the third condition above. Given $\bm{a}$ which comes from a choice of $w \in W'_{D_n}$ and $\beta \in \Delta_{D_n}$ such that $\beta \neq \beta_{n-1}$, the fact that $\beta \neq \beta_{n-1}$ is equivalent to the simple roots of $A_{2n-1}$ corresponding to $\beta$ not intersecting $\{\bm{e}_{n-1}-\bm{e}_n, \bm{e}_{n+1}-\bm{e}_{n+2}\}$. This is equivalent to $\sum_{i=0}^{m^A} a_i \neq n-1$. With these changes, the proof of Lemma~\ref{lem:B-wt-corresp} proves the result.
\end{proof}
We now proceed with the proof of Theorem \ref{thm: main}, first assuming that $\rho$ is valued in $\{0,2\}$. Recall from \cite{BalaCarter} that $\rho$ is distinguished precisely when one of the following holds:
\begin{enumerate}
    \item $\rho(\beta_{n-1}) = \rho(\beta_n) = 2$ and $N_0=1$, $N_{i-1} \leq N_i \leq N_{i-1}+1$ for $1 \leq i \leq m-2$ and $N_{m-2} \leq 2$,
    \item $\rho(\beta_{n-1}) = \rho(\beta_n) = 0$ and $N_0=1$, $N_{i-1} \leq N_i \leq N_{i-1}+1$ for $1 \leq i \leq m-1$ and
    \[
        N_m = \begin{cases}
            \frac{N_{m-1}+1}{2}, & N_{m-1} \text{ odd} \\
            \frac{N_{m-1}}{2}, & N_{m-1} \text{ even}
        \end{cases},
    \]
    which is equivalent to $N_{m-1}^A \leq N_m^A \leq N_{m-1}^A+1$.
\end{enumerate}
We now deduce that in each distinguished case, $\zeta_{D_n}(w)_{\Delta}$ is strictly positive for all $w \in W'_{D_n}$. For the $\rho(\beta_{n-1}) = \rho(\beta_n) = 0$ case, this follows from applying Lemma~\ref{lem:matrix-ineq} to \eqref{eqn: Dn final formula 00 1} and \eqref{eqn: Dn final formula 00 2} since we have $\zeta_{D_n}(w)_\beta \geq 0$ in such cases, with equality if and only if $a_m = 0$ and $\bm{u}_i - \bm{u}_{i-1} \in \{\bm{0},\bm{e}_1,\bm{e}_2\}$ for all $0 \leq i \leq m$ and $\bm{u}_i - \bm{u}_{i-1} \in \{\bm{0},-\bm{e}_1,-\bm{e}_2\}$ for all $m+1 \leq i \leq 2m+1$. This implies $a_i = 0$ for all $0 \leq i \leq 2m$, which does not correspond to a choice of $w$ and $\beta$ by Lemma~\ref{lem:D-wt-corresp}. For the $\rho(\beta_{n-1}) = \rho(\beta_n) = 2$ case, essentially the same argument works after adjusting the indexing (note that $N_{m-1}^A = 2$ in this case) by applying Lemma~\ref{lem:matrix-ineq} to \eqref{eqn: Dn final formula 22 1} and \eqref{eqn: Dn final formula 22 2}, and then using Lemma~\ref{lem:D-wt-corresp}.

We now prove the converse. First we assume $\rho(\beta_{n-1}) = \rho(\beta_n) = 0$, and note that this implies $N_m \geq 2$ (and thus $N_m^A \geq 4$). Suppose for some $0 \leq s \leq m$ we have $N^A_s \leq N^A_{s-1} - 1$ and let $s$ be minimal with this property and observe $N_{-1} = 0$ implies $s \geq 1$. Set $a_i = 1$ for $0 \leq i \leq s-1$, and set $a_i = 0$ for all other values of $i$. Thus $a_i + a_{2m-i} \leq 1 \leq N_i^A$ for all $i$, and $\sum_{i=0}^{2m} a_i \leq n-2$. Therefore $\bm{a}$ corresponds to a choice of $w$ and $\beta$ by Lemma~\ref{lem:D-wt-corresp}. Up to a factor of $2$ when $\beta = \beta_n$ (compare \eqref{eqn: Dn final formula 00 1} and \eqref{eqn: Dn final formula 00 2}), we now compute
\[
\begin{split}
    \zeta_{D_n}(w)_\beta &= -\frac{1}{2} (\bm{u}_0-\bm{u}_{-1})^\top
        \begin{bmatrix}
            0 & 1 \\ 1 & 0
        \end{bmatrix}
        (\bm{u}_0-\bm{u}_{-1}) \\
        &- \frac{1}{2} (\bm{u}_s-\bm{u}_{s-1})^\top
        \begin{bmatrix}
            0 & 1 \\ 1 & 0
        \end{bmatrix}
        (\bm{u}_s-\bm{u}_{s-1}) + 2a_m \\
        &= -(N_0^A-1) + (N_s^A - N_{s-1}^A + 1) + 0 \\
        &\leq 0.
\end{split}
\]

The remaining case to consider when $\rho(\beta_{n-1}) = \rho(\beta_n) = 0$ is where $N^A_i \geq N^A_{i-1}$ for all $0 \leq i \leq m$, and for some $0 \leq s \leq m$ we have $N^A_s \geq N^A_{s-1}+2$. Let $s$ be maximal with this property. Thus $N^A_j \geq 2$ for all $s \leq j \leq 2m-s$. Now set $a_i = 0$ for $i < s$ and $i > 2m-s$, and set $a_i = 1$ for $s \leq i \leq 2m-s$. Thus $a_i + a_{2m-i} = 2 \leq N_i^A$ for all $s \leq i \leq 2m-s$ and $a_i + a_{2m-i} = 0$ otherwise. Since $N_m \geq 2$, if $\sum_{i=0}^{2m} a_i = n-1$
then $N_m = 2$ and $N_i = a_i+a_{2m-i}$ for all $0 \leq i \leq m-1$, which implies $s = 0$ and $N_i = 2$ for all $0 \leq i \leq m-1$. But then $N_m^A = 4 = N_{m-1}^A + 2$, contradicting maximality of $s$. Therefore $\bm{a}$ corresponds to a choice of $w$ and $\beta$ by Lemma~\ref{lem:D-wt-corresp}. Up to a factor of $2$ when $\beta = \beta_n$ (compare \eqref{eqn: Dn final formula 00 1} and \eqref{eqn: Dn final formula 00 2}), we now compute
\[
\begin{split}
    \zeta_{D_n}(w)_\beta &= -\frac{1}{2} (\bm{u}_s-\bm{u}_{s-1})^\top
        \begin{bmatrix}
            0 & 1 \\ 1 & 0
        \end{bmatrix}
        (\bm{u}_s-\bm{u}_{s-1}) \\
        &- \frac{1}{2} (\bm{u}_{2m-s+1}-\bm{u}_{2m-s})^\top
        \begin{bmatrix}
            0 & 1 \\ 1 & 0
        \end{bmatrix}
        (\bm{u}_{2m-s+1}-\bm{u}_{2m-s}) + 2a_m \\
        &= -(N_s^A-N_{s-1}^A-1) + (N_{s-1}^A - N_s^A + 1) + 2 \\
        &\leq -2(2-1) + 2 = 0.
\end{split}
\]

For the case of $\rho(\beta_{n-1}) = \rho(\beta_n) = 2$, essentially the same argument works as in the $\rho(\beta_{n-1}) = \rho(\beta_n) = 0$ case after re-indexing ($m \to m-1$) and using Lemma~\ref{lem:D-wt-corresp}. The main difference comes in the fact that $N_{m-1} = 1$ (and thus $N_{m-1}^A = 2$) in the application of Lemma~\ref{lem:D-wt-corresp}. For the first case where $s \leq m-1$ is minimal such that $N_s^A \leq N_{s-1}^A - 1$, we have $N_{s-1}^A \geq 2$ which implies
$\sum_{i=0}^{2m-2} a_i \leq n-(1+1) = n-2$. For the second case where $N_i^A \geq N_{i-1}^A$ for all $0 \leq i \leq m-1$ and $s$ is maximal such that $N_s^A \geq N_{s-1}^A + 2$, we must have $N_i = 2$ for all $0 \leq i \leq m-2$ and thus $s=0$. This implies $\sum_{i=0}^{2m-2} a_i = n$. Therefore in any case $\bm{a}$ corresponds to a choice of $w$ and $\beta$ by Lemma~\ref{lem:D-wt-corresp}.

Now we assume $\rho(\beta_{n-1}) = 0$ and $\rho(\beta_n) = 2$. Let $s \leq m-1$ be minimal such that $N_j^A \geq 2$ for all $s \leq j \leq m-1$. (Since $N_{m-1}^A = N_m^A \geq 2$ in this case, this is always possible.) Set $a_i = 1$ for all $s \leq i \leq m-1$, and set $a_i = 0$ for all other values of $i$. Thus $a_i + a_{2m-1-i} \leq 1 \leq N_i^A$ for all $i$, and $\sum_{i=0}^{2m-1} a_i \leq n-2$.
Therefore $\bm{a}$ corresponds to a choice of $w$ and $\beta$ by Lemma~\ref{lem:D-wt-corresp}. Up to a factor of $2$ when $\beta = \beta_n$ (compare \eqref{eqn: Dn final formula 02 1} and \eqref{eqn: Dn final formula 02 2}), we compute
\[
\begin{split}
    \zeta_{D_n}(w)_\beta &= -\frac{1}{2} (\bm{u}_s-\bm{u}_{s-1})^\top
        \begin{bmatrix}
            0 & 1 \\ 1 & 0
        \end{bmatrix}
        (\bm{u}_s-\bm{u}_{s-1}) \\
        &- \frac{1}{2} (\bm{u}_m-\bm{u}_{m-1})^\top
        \begin{bmatrix}
            0 & 1 \\ 1 & 0
        \end{bmatrix}
        (\bm{u}_m-\bm{u}_{m-1}) - (a_{m-1}+a_m) \\
        &= -(N_s^A-N_{s-1}^A) + 1 \leq 0. \\
\end{split}
\]

We now consider the case where $\rho$ is valued in $\Q_{\geq 0}$. Up to possibly replacing $\rho$ with $\theta(\rho)$, we may assume that $\rho(\bm{e}_{n-1} + \bm{e}_n) \geq \rho(\bm{e}_{n-1} - \bm{e}_{n})$ so that $f^{-1}(\rho)$ is valued in $\Q_{\geq 0}$. Similarly to the $B$- and $C$-type cases, we can find a $w \in W'_{D_n}$ such that $w(V_2)$ comes from a weight function $\rho'$ of $D_n$ such that $f^{-1}(\rho')$ is valued in $\{0,1,2,3\}$. If $f^{-1}(\rho')$ has a simple root of value $3$, then $\rho'$ has a simple root of value $\geq 3$ and hence cannot be non-degenerate by testing at $w=1$. We may as well assume that at least one simple root of $A_{2n-1}$ takes value $1$, otherwise $\rho'$ is valued in $\{0,2,4\}$. And, if $\rho'(\beta_n) = 4$ then every positive root involving $\beta_n$ has weight at least $4$, thus making $\rho'$ degenerate. Thus we may assume that $f^{-1}(\rho')$ is valued in $\{0,1,2\}$ with the groupings near the center separated by $1$s. Hence, $\rho'$ is valued in $\{0,1,2\}$. Moreover, we only need to consider the case with exactly $2$ self-dual components, which can only happen if there are an odd number of groupings and hence that the central simple $A$ root takes value $0$. This means that $\rho'$ takes the same values on the tail roots. The correction term for a non-tail simple root is of the form $2a_m - (a_{m-1} + a_{m+1})$ or $2a_{m-1}  - (a_{m-2}+a_m)$ depending on whether $\rho'(\beta_n)= 0$ or $\rho'(\beta_n) = 1$. (Note that $\rho'(\beta_n) \neq 2$ or else $\rho'$ is valued in $\{0,2\}$.) For a tail term, the correction is half as much, but we will not need this below.

We again consider the two components $X_1, X_2$ of the graph associated to $f^{-1}(\rho')$ where $X_1$ contains the central grouping. Let $c$ index the central block, so that block $c$ is contained in $X_1$ and the neighboring blocks are contained in $X_2$. Set
\[
    a_i = \begin{cases}
        1, & i \in X_2, ~ i < c \\
        0, \text{otherwise.}
    \end{cases}
\]

The previous formulas give
\[
    \zeta_{\rho'}(w)_\beta = 1 - N_{\min\{X_2\}}^A \leq 0
\]
as long as $\sum_{i=0}^{m^A} a_i \leq n-2$, so that $\bm{a}$ is admissible and Lemma~\ref{lem:D-wt-corresp} applies. (This $\leq n-2$ bound is also why we do not need to worry about the half correction for the tail term.)

To handle admissibility, note that if $\rho'(\beta_{n-1}) = \rho'(\beta_n) = 0$ then the central grouping has at least $4$ vertices, implying $\sum_{i=0}^{m^A} a_i \leq n-2$. If $\rho'(\beta_{n-1}) = \rho'(\beta_n) = 1$ then we only obtain $\sum_{i=0}^{m^A} a_i \leq n-1$. But equality here implies we in fact have $\rho'(\beta_1,\ldots,\beta_{n-2},\beta_{n-1},\beta_n) = (2,\ldots,2,1,1)$. In this case we have $V_2 = \{\beta_1,\ldots,\beta_{n-2}\}$ and $V_0 = \varnothing$, and thus $\zeta_{\rho'}(1)_{\beta_{n-1}} = \zeta_{\rho'}(1)_{\beta_n} = 0$. In any case, we have that $\rho'$ is degenerate.

\section{Exceptional Groups}

For the exceptional cases, one can attempt a reduction to $A$-type using quasi-minuscule representations, but these reductions seem to involve significant casework, especially for larger groups. To give an indication of what is involved, we briefly consider the simplest case of $G_2$.

We have a $7$-dimensional quasi-minuscule representation of $G_2$ and this indicates we can define a map $f: \R^7 \to \R^3$ via
\[
    f(\bm{e}_i) := \begin{cases}
        \bm{e}_1-\bm{e}_3, & i = 1 \\
        \bm{e}_2-\bm{e}_3, & i = 2 \\
        \bm{e}_1-\bm{e}_2, & i = 3 \\
        0, & i = 4 \\
        -(\bm{e}_1-\bm{e}_2), & i = 5 \\
        -(\bm{e}_2-\bm{e}_3), & i = 6 \\
        -(\bm{e}_1-\bm{e}_3), & i = 7.
    \end{cases}
\]
Let $\bm{e}_i-\bm{e}_{i+1}$ be the simple roots of $A_6$, and let $\alpha = \bm{e}_1-\bm{e}_2$ and $\beta = -\bm{e}_1 + 2\bm{e}_2 - \bm{e}_3$ be the simple roots of $G_2$. Then the positive roots of $G_2$ are $\alpha,\beta,\beta+\alpha,\beta+2\alpha,\beta+3\alpha,2\beta+3\alpha$. Note that $f$ maps simple roots to simple roots via
\[
    f(\bm{e}_i-\bm{e}_{i+1}) = \begin{cases}
        \alpha, & i \in \{1,3,4,6\} \\
        \beta, & i \in \{2,5\}.
    \end{cases}
\]
This map is positive and root surjective. Moreover, we can define $\phi: W_{G_2} \rightarrow W_{A_6}$ by $s_{\alpha} \mapsto (12)(35)(67) , s_{\beta} \mapsto (23)(56)$ and $\phi$ is compatible with $f$. With this, we can define $H_i^A$ and $N_i^A$ and $a_i$ as in the classical cases, and attempt to utilize the results of Section~\ref{sec:A-reduction}. However, unlike in the classical group cases, the values of $a_i$ do not determine the value of $\zeta_{G_2}(w)_\gamma$ for a given choice of $w \in W_{G_2}$ and $\gamma \in \Delta_{G_2}$. This means that handling the case of $G_2$ requires a significant amount of casework, which is likely even more true of the other exceptional groups.

In light of these complications, we decided instead to check these cases directly using a computer. We used the root systems library provided by \cite{sagemath}, and our Python code can be found in \cite{pos-check-code}. On all exceptional groups except for $E_8$, the code takes a few hours to run on a personal laptop. On the other hand, $E_8$ required a week on a supercomputer. The most time is spent on the distinguished $\rho$, for which positivity of the coefficient values must be checked for every Weyl group element. For non-distinguished $\rho$ the code runs faster since one only needs to find a single Weyl element for which some coefficient is non-positive. A full listing of such Weyl elements for all non-distinguished $\rho$ can be found in \cite{pos-check-code}.

\subsection{The case of general $\rho$}

We now consider the case of general weight function $\rho$ valued in $\Q_{\geq 0}$. Letting $\Delta = \{\alpha_1,\ldots,\alpha_n\}$ be the simple roots and $\Phi^{>0} = \{\beta_1,\ldots,\beta_N\}$ be the positive roots, there exist $a_{ij} \in \Z_{\geq 0}$ such that
\[
    \beta_i = \sum_{j=1}^n a_{ij} \cdot \alpha_j
\]
for all $1 \leq i \leq N$. Given any general weight function $\rho$, the positive roots are partitioned via $\Phi^{>0} = V_2 \sqcup V_0 \sqcup V_*$, where $V_*$ consists of all other positive roots. (Here we are abusing notation slightly by considering $V_0$ to only consist of positive roots.) We note that
\[
    \begin{cases}
        2 = \sum_{j=1}^n a_{ij} \cdot \rho(\alpha_j), & \forall \beta_i \in V_2 \\
        0 = \sum_{j=1}^n a_{ij} \cdot \rho(\alpha_j), & \forall \beta_i \in V_0 \\
        \{0,2\} \not\ni \sum_{j=1}^n a_{ij} \cdot \rho(\alpha_j), & \forall \beta_i \in V_* \\
    \end{cases}
\]
by definition of $V_2$ and $V_0$.
\begin{proposition} \label{prop:distinguished-iter}
    Let $A$ be the matrix with entries $a_{ij}$ as defined above for $1 \leq i \leq N$ and $1 \leq j \leq n$. Let $\rho: \Delta \to \Q_{\geq 0}$ be a non-degenerate weight function, and define $x_\rho := (\rho(\alpha_1), \rho(\alpha_2), \ldots, \rho(\alpha_n))$. Then there exists a non-singular $n \times n$ submatrix $B$ of $A$ such that $x_\rho = B^{-1} \bm{2}$, where $\bm{2}$ is the length-$n$ all-twos vector.
\end{proposition}
\begin{proof}
    Since $\rho$ is non-degenerate, we have that $\bP_{\cusp}$ is full-dimensional by Lemma~\ref{lem: fulldimlemma}, which implies $\bP_{\naive} \supset \bP_{\cusp}$ is full-dimensional. Thus $V_2$ contains a basis of $\R^\Delta$, and therefore the associated rows of $A$ form a basis of $\R^n$ (since $\Delta$ is a basis of $\R^n$). Choosing $B$ to be this set of $n$ rows of $A$ implies the result.
\end{proof}

\begin{corollary} \label{cor:distinguished-count}
    There are at most $\binom{N}{n}$ distinct weight functions $\rho: \Delta \to \Q_{\geq 0}$ which are non-degenerate.
\end{corollary}
\begin{proof}
    Proposition~\ref{prop:distinguished-iter} gives an injection from non-degenerate weight functions to a choice of $n$ rows of $A$, and thus the number of non-degenerate weight functions is bounded by $\binom{N}{n}$.
\end{proof}

The previous results then immediately imply an algorithm for determining all non-degenerate weight functions $\rho: \Delta \to \Q_{\geq 0}$:
\begin{enumerate}
    \item Iterate through all $n \times n$ submatrices $B$ of $A$.
    \item For each such $B$, check if $B$ is invertible.
    \item If $B$ is invertible then check if $x := B^{-1}\bm{2}$ is non-negative.
    \item If $x$ is non-negative then define $\rho(\alpha_j) = x_j$ and check if $\rho$ is non-degenerate.
\end{enumerate}
This is precisely the algorithm we use to determine all possible rational non-degenerate weight functions $\rho$ for the exceptional groups. For most of the groups, this algorithm runs within a few hours on a personal laptop. As above, for $E_8$ the number $\binom{N}{n} = \binom{120}{8} \approx 8.40 \times 10^{11}$ is quite large. Thus we again utilized supercomputing resources to run this algorithm, and our Python code can be found in \cite{pos-check-code}. Using this method, we did not find any new rational non-degenerate weight functions $\rho$.

\printbibliography

\end{document}